\documentclass{svmult}

\pdfoutput=1

\usepackage[utf8]{inputenc}

\usepackage{mathptmx}
\usepackage{helvet}
\usepackage{courier}
\usepackage{graphicx}
\usepackage{multicol}
\usepackage{footmisc}
\usepackage{xcolor}
\usepackage{amsmath, amssymb}
\usepackage{mathtools}
\usepackage[ruled]{algorithm2e}
\usepackage[square]{natbib}

\renewcommand{\vec}[1]{\textbf{\textit{#1}}}

\definecolor{dark_blue}{rgb}{0.1,0.1,0.44}

\numberwithin{theorem}{section}

\spnewtheorem{corollary}[theorem]{Corollary}{\bfseries}{\itshape}
\spnewtheorem{definition}[theorem]{Definition}{\bfseries}{\rmfamily}
\spnewtheorem{lemma}[theorem]{Lemma}{\bfseries}{\itshape}

\allowdisplaybreaks

\begin{document}

\title*{Fast Fourier Transforms for Spherical Gauss-Laguerre Basis Functions}
\author{J\"urgen Prestin and Christian W\"ulker}
\institute{J\"urgen Prestin \at Institute of Mathematics, University of L\"ubeck, Ratzeburger Allee 160, 23562 L\"ubeck, Germany
\and Christian W\"ulker \at Institute of Mathematics, University of L\"ubeck, Ratzeburger Allee 160, 23562 L\"ubeck, Germany, 
\email{wuelker@math.uni-luebeck.de}; a MATLAB implementation of the algorithms presented in this paper is available at 
https://github.com/cwuelker/SGLPack}

\maketitle

\abstract{Spherical Gauss-Laguerre (SGL) basis functions, i.e., normalized functions of the type $L_{n-l-1}^{(l + 1/2)\!}
(r^2) \hspace*{1pt} r^{\hspace*{0.5pt}l} \hspace*{1pt} Y_{lm}(\vartheta,\varphi)$, $|m| \leq l < n \in \mathbb{N}$, 
$L_{n-l-1}^{(l + 1/2)\!}$ being a generalized Laguerre polynomial, $Y_{lm}$ a spherical harmonic, constitute an orthonormal 
basis of the space $L^{2\!}$ on $\mathbb{R}^{3\!}$ with Gaussian weight $\exp(-r^{2})$. These basis functions are used 
extensively, e.g., in biomolecular dynamic simulations. However, to the present, there is no reliable algorithm 
available to compute the Fourier coefficients of a function with respect to the SGL basis functions in a fast way. This 
paper presents such generalized FFTs. We start out from an SGL sampling theorem that permits an exact computation of 
the SGL Fourier expansion of bandlimited functions. By a separation-of-variables approach and the employment of a fast 
spherical Fourier transform, we then unveil a general class of fast SGL Fourier transforms. All of these algorithms have 
an asymptotic complexity of $\mathcal{O}(B^{4})$, $B$ being the respective bandlimit, while the number of sample points 
on $\mathbb{R}^{3\!}$ scales with $B^{3\!}$. This clearly improves the naive bound of $\mathcal{O}(B^{7})$. At the same 
time, our approach results in fast inverse transforms with the same asymptotic complexity as the forward transforms. We 
demonstrate the practical suitability of our algorithms in a numerical experiment. Notably, this is one of the first 
performances of generalized FFTs on a non-compact domain. We conclude with a discussion, including the layout of a true 
$\mathcal{O}(B^{3} \log^{2\!} B)$ fast SGL Fourier transform and inverse, and an outlook on future developments.}

\section{Introduction}\label{sec:intro}
\begin{sloppypar}
Since its popularization by \cite{fft}, the \emph{Fast Fourier Transform} (FFT) on the unit circle $\mathbb{T}$ 
and its inverse (iFFT) have been generalized to several other domains and corresponding sets of basis functions. 
For example, many applications in signal processing and data analysis nowadays benefit from an extension of the 
univariate FFTs to the $d$-dimensional Torus $\mathbb{T}^{d\!}$\!\hspace*{1pt} ($d \!>\! 1$), where multivariate 
trigonometric polynomials are used in analogy to the univariate case (see, e.g., \cite[Sect.\ 2]{mfft}). 
Another example is the two-dimensional unit sphere $\mathbb{S}^{2\!}$. Here, the spherical harmonics are 
used as an orthonormal basis of the space $L^{2}(\mathbb{S}^{2})$ of functions square-integrable over 
$\mathbb{S}^{2\!}$\!\hspace*{1pt} (see \citep{sfft, sfft2, sfft3, sfft4}, for instance). This has also initiated 
the development of fast Fourier transforms on the tree-dimensional rotation group $\textnormal{SO}(3)$, where 
the spherical harmonics are replaced by so-called Wigner-$D$ functions (see \citep{sofft, sofft2, sofft3}, for 
example). Recently, certain combinations of spherical harmonics, generalized Laguerre polynomials, and 
an exponential radial decay factor were used as orthonormal basis functions of the space $L^{2}(\mathbb{B}^{3})$ 
of square-integrable functions on the three-dimensional unit ball $\mathbb{B}^{3}$; a fast Fourier transform was 
developed in this setting as well (see \citep{ball} for more information).
\end{sloppypar}

In this work, we introduce fast Fourier transforms on the entire three-dimensional real space $\mathbb{R}^{3\!}$. 
On the one hand, this extends the above collection of domains in a natural direction; on the other hand, due to 
the non-compactness of $\mathbb{R}^{3\!}$, we find ourselves in a somewhat new situation.

Of course, the non-compactness of the underlying domain has to be accounted for. While it is conceivable to 
consider basis functions that exhibit an appropriate decay behavior, in this work, we endow the space $L^{2}
(\mathbb{R}^{3})$ with the Gaussian weight function $\exp(-|\cdot|^{2})$, where $|\cdot|$ denotes the 
standard Euclidean norm (such weight function is also referred to as a \emph{multivariate Hermite 
weight} in literature). In particular, we consider the weighted $L^{2\!}$ 
space
\begin{equation*}
H \hspace*{0pt}\coloneqq\hspace*{0pt} \left\lbrace f : \mathbb{R}^{3\!} \to \mathbb{C} 
\hspace*{0.75pt}:\hspace*{0.75pt} f~\textnormal{(Lebesgue) measurable and}\hspace*{1pt}
\int_{\mathbb{R}^{3}}\! |f(\vec{x})|^{2\!} \hspace*{1pt} \exp(-|\vec{x}|^{2}) \hspace*{1pt} \mathrm{d}\vec{x} < \infty 
\right\rbrace\!\!,
\end{equation*}
endowed with the inner product
\begin{equation}\label{eq:inner_product}
\langle f, g \rangle_{\!\hspace*{0.5pt}H} \hspace*{2pt}\coloneqq\hspace*{2pt} \!\int_{\mathbb{R}^{3}} f(\vec{x}) \hspace*{1pt} \overline{g(\vec{x})} \hspace*{0pt} \exp(-|\vec{x}|^{2}) \hspace*{1pt} \mathrm{d}\vec{x}, ~~~~~~ f, g \in H,
\end{equation}
and induced norm $\|\cdot\|_{H} \coloneqq \sqrt{\langle\cdot,\cdot\rangle_{\!\hspace*{0.5pt}H}}$.

\begin{sloppypar}
A crucial feature of the space $H$ is that it allows to work with such structurally simple functions as 
\emph{polynomials}. Particularly, as recently noted by \citet[Sect.\ 1]{density}, we have the following 
result, essential for everything to follow:
\end{sloppypar}

\begin{theorem}\label{thm:density}
The class of (complex-valued) polynomials on \hspace*{0.5pt}$\mathbb{R}^{3\!}$ is dense in $H$, i.e., 
any function $f \in H$ can be approximated arbitrarily well by polynomials with respect to $\|\cdot\|_{H}$.
\end{theorem}

Having this in mind, it appears natural to employ appropriately normalized orthogonal polynomials 
as an orthonormal basis of the Hilbert space $H$. In view of this, however, we note that orthogonal 
polynomials in $H$ are not unique, as we should expect in the univariate setting. In fact, a review 
of the relevant literature reveals several different variants of such, arising from different construction 
approaches (see, e.g., \cite[Sect.\ 5.1.3]{dunkl_xu}).

By a \emph{separation-of-variables} approach, \citet{gto} constructed particular orthogonal polynomials in $H$
from the well-known spherical harmonics (Definition \ref{def:spherical_harmonics}) and generalized Laguerre 
polynomials (Theorem \ref{thm:generalized_laguerre}). We call these \emph{spherical Gauss-Laguerre} (SGL) 
\emph{basis functions} (the term `Gaussian' is to account for the Gaussian weight on $H$).

\begin{definition}[\textbf{SGL basis functions}]\label{def:sgl}
The SGL basis function of \emph{degree} $n \in \mathbb{N}$ and \emph{orders} $l \in \{0,\dots,n-1\}$ and 
$m \in \{-l,\dots,l\}$ is defined in spherical coordinates (see Section \ref{sec:sgl}) as
\begin{equation}\label{eq:sgl}
H_{nlm} : \mathbb{R}^{3\!} \to \mathbb{C}, ~~~~~~ H_{nlm} (r,\vartheta,\varphi) \hspace*{2pt}\coloneqq\hspace*{2pt} N_{nl} \hspace*{1pt} R_{nl} (r) \hspace*{1pt} Y_{lm}(\vartheta,\varphi),
\end{equation}
where $N_{nl}$ is a normalization constant, 
\begin{equation*}
N_{nl} \hspace*{2pt}\coloneqq\hspace*{2pt} \sqrt{\frac{2 (n - l - 1)!}{\Gamma(n + 1/2)}}, 
\end{equation*}
$Y_{lm}$ is the spherical harmonic of degree $l$ and order $m$, while the radial part $R_{nl}$ is defined 
as
\begin{equation*}
R_{nl}(r) \hspace*{2pt}\coloneqq\hspace*{2pt} L_{n - l - 1}^{(l + 1/2)\!}(r^{2}) \hspace*{1pt} r^{l\!},
\end{equation*}
$L_{n - l - 1}^{(l + 1/2)}$ being a generalized Laguerre polynomial.
\end{definition}

By construction of the SGL basis functions, these polynomials are orthonormal in $H$ and span the space 
of all polynomials on $\mathbb{R}^{3\!}$ (see Section \ref{sec:sgl}). The completeness of this orthonormal 
system thus follows from Theorem \ref{thm:density}.

\begin{corollary}\label{thm:sgl_2}
The SGL basis functions $H_{nlm}$ constitute an orthonormal basis (i.e., a complete orthonormal system) in 
$H$. In particular, for any $f \in H$, the \emph{Fourier partial sums}
\begin{equation}\label{eq:SGL_partial_sum}
\sum_{n = 1}^{B} \hspace*{3.5pt} \sum_{l = 0}^{n - 1} \sum_{m = -l}^{l}\! \langle f, H_{nlm} \rangle_{\!\hspace*{0.5pt}H} \hspace*{1pt} H_{nlm}, ~~~~~~ B \in \mathbb{N},
\end{equation}
converge to $f$ in the norm of \hspace*{0.75pt}$H$\!\hspace*{0.75pt} as \hspace*{0.15pt}$B$ 
approaches $\infty$.
\end{corollary}

In this paper, we present a general class of algorithms for the efficient numerical computation of the SGL 
Fourier coefficients $\hat{f}_{nlm} \coloneqq \langle f, H_{nlm} \rangle_{\!\hspace*{0.5pt}H}$ in 
\eqref{eq:SGL_partial_sum}\:--\:that is, we present \emph{fast SGL Fourier transforms}. 
As is commonly done in generalized FFTs, we develop our algorithms starting out from a concomitant quadrature 
formula, so that these algorithms are exact (in exact arithmetics) for bandlimited functions (see Section 
\ref{sec:sglft}). Inspired by the construction of the SGL basis functions, our approach is based on 
a separation of variables, separating the radius $r$ from the angles $\vartheta$ and $\varphi$. For the radial 
part of our fast transforms, we introduce the \emph{discrete $R$ transform} (Section \ref{subsubsec:DRT}). The 
spherical part of our transforms is constituted by a fast spherical Fourier transform, 
i.e., a generalized FFT for the spherical harmonics. Notably, our approach also results in fast inverse 
transforms with the same asymptotic complexity as the forward transforms: All of our fast algorithms have an 
asymptotic complexity of $\mathcal{O}(B^{4})$, $B$ being the respective bandlimit, while the number of sample 
points on $\mathbb{R}^{3\!}$ scales with $B^{3\!}$. This clearly improves the naive bound of $\mathcal{O}(B^{7})$.

Applications of our fast algorithms arise, for example, in the simulation of biomolecular recognition processes, 
such as \emph{protein-protein} or \emph{protein-ligand docking} (see Section \ref{sec:discussion}).

The rest of this paper is organized as follows: In Section \ref{sec:sgl}, we review the construction 
of the SGL basis functions. This section is optional to the reader interested solely in our fast algorithms. 
Subsequently, in Section \ref{sec:sglft}, we develop fast SGL Fourier transforms. The resulting 
algorithms are tested in a prototypical numerical experiment in Section \ref{sec:experiments}. In Section
\ref{sec:discussion}, we discuss the results, draw final conclusions, and give an outlook on future 
developments. We also include the layout of a true $\mathcal{O}(B^{3} \log^{2\!} B)$ fast SGL Fourier 
transform and inverse.

\section{Spherical Gauss-Laguerre (SGL) basis functions}\label{sec:sgl}
As mentioned above, the SGL basis functions of Definition \ref{def:sgl} arise from a particular construction approach by \cite{gto}. 
This approach comprises multiple steps. The first step is the introduction of \emph{spherical coordinates}. We define these as 
\emph{radius} $r \in [0,\infty)$, \emph{polar angle} $\vartheta \in [0,\pi]$, and \emph{azimuthal angle} $\varphi \in [0,2\pi)$, 
being connected to Cartesian coordinates $x$, $y$, and $z$, via
\begin{align*}
x \hspace*{2pt}&=\hspace*{2pt} r \hspace*{1pt} \sin\vartheta \hspace*{1pt} \cos\varphi,\\
y \hspace*{2pt}&=\hspace*{2pt} r \hspace*{1pt} \sin\vartheta \hspace*{1pt} \sin\varphi,\\
z \hspace*{2pt}&=\hspace*{2pt} r \hspace*{1pt} \cos\vartheta.
\end{align*}  

In the following, with a slight abuse of notation, we write $f(\vec{x}) = f(r,\vartheta,\varphi)$ if $(r,\vartheta,\varphi)$ 
are the spherical coordinates of the point $\vec{x} = (x,y,z) \in \mathbb{R}^{3\!}$, in which case we simply write 
$\vec{x} = (r,\vartheta,\varphi)$. This allows the inner product \eqref{eq:inner_product} to be 
rewritten as
\begin{equation}\label{eq:inner_product_spherical_coordinates}
\langle f, g \rangle_{\!\hspace*{0.5pt}H} \hspace*{0pt}=\hspace*{0pt} \!\int_{0}^{\infty} \! \Bigl\lbrace \hspace*{1pt}\int_{0}^{\pi} \! \int_{0}^{2\pi}\!\! f(r,\vartheta,\varphi) \hspace*{1pt} \overline{g(r,\vartheta,\varphi)} \hspace*{2pt} \mathrm{d}\varphi \hspace*{1pt} \sin\vartheta \hspace*{2pt} \mathrm{d}\vartheta \Bigr\rbrace \hspace*{1pt} r^{2} \hspace*{1pt} \mathrm{e}^{-r^{2\!}} \hspace*{1pt} \mathrm{d}r, \:~~ f, g \in H.
\end{equation}
Note that the integration range $[0,\pi] \times [0,2\pi)$ of the two inner integrals above can be identified with the unit sphere 
$\mathbb{S}^{2\!}$.

\begin{sloppypar}
The next step is a \emph{separation of variables}. In particular, \citeauthor{gto} make the product ansatz 
\begin{equation}\label{eq:product_ansatz}
p(\vec{x}) \hspace*{2pt}=\hspace*{2pt} R(r) \hspace*{1pt} S(\vartheta,\varphi), ~~~~~~ \vec{x} = (r,\vartheta,\varphi) \in \mathbb{R}^{3\!},
\end{equation}
for each orthogonal polynomial $p$ to be constructed. 
\end{sloppypar}

Of course, the radial part $R$ and the spherical part $S$ should be \emph{polynomial} on $[0,\infty)$ and $\mathbb{S}^{2\!}$ 
(by which we mean the restriction of a polynomial on $\mathbb{R}^{3\!}$ to $\mathbb{S}^{2}$), respectively. 
Furthermore, it is desirable that each two orthogonal polynomials $p_{j}$ and $p_{k}$ satisfy separate orthogonality relations 
with respect to the radius and on the sphere,
\begin{align}
\int_{0}^{\infty}\!\! R_{j}(r) \hspace*{0.5pt} \overline{R_{k}(r)} \hspace*{2pt} r^{2} \hspace*{1pt} \mathrm{e}^{-r^{2\!}} \hspace*{1pt} \mathrm{d}r \hspace*{2pt}&=\hspace*{2pt} \delta_{jk},\label{eq:O1}\\
\int_{0}^{\pi} \! \int_{0}^{2\pi}\!\! S_{j}(\vartheta,\varphi) \hspace*{1pt} \overline{S_{k}(\vartheta,\varphi)} \hspace*{2pt} \mathrm{d}\varphi \hspace*{1pt} \sin\vartheta \hspace*{2pt} \mathrm{d}\vartheta \hspace*{2pt}&=\hspace*{2pt} \delta_{jk},\label{eq:O2}
\end{align}
denoting by $\delta_{jk}$ the standard \emph{Kronecker symbol}, being $1$ if $j = k$ and $0$ otherwise. The  
property $\langle p_{j}, p_{k} \rangle_{H\!} = \delta_{jk}$, i.e., the orthonormality of the SGL basis functions, 
then follows by \eqref{eq:inner_product_spherical_coordinates}.

The above separation approach allows the radial part $R$ and the spherical part $S$ in \eqref{eq:product_ansatz} to be 
constructed almost independently from each other. We begin with the spherical part $S$, for which solely the spherical 
harmonics are required.

\begin{definition}\label{def:spherical_harmonics}
The \emph{spherical harmonic} of \emph{degree} $l \in \mathbb{N}_{0}$ and \emph{order} $m \in \{-l,\dots,l\}$ is defined as
\begin{equation}\label{eq:spherical_harmonics}
Y_{lm} : \mathbb{S}^{2\!} \to \mathbb{C}, ~~~~~~ Y_{lm} (\vartheta, \varphi) \hspace*{2pt}\coloneqq\hspace*{2pt} \sqrt{\frac{(2l+1)}{4\pi}\frac{(l-m)!}{(l+m)!}} \hspace*{1pt} P_{lm}(\cos \vartheta)\hspace*{1pt} \mathrm{e}^{\mathrm{i}m\varphi\!},
\end{equation}
where $P_{lm\!}$ denotes the \emph{associated Legendre polynomial} of degree $l$ and order $m$ \cite[Eqs.\ 8.6.6 and 8.6.18]{abramowitz}:
\begin{equation*}
P_{lm} : [-1,1] \to \mathbb{R}, ~~~~~~ P_{lm}(t) \hspace*{2pt}\coloneqq\hspace*{2pt} \frac{(-1)^{m}}{2^{l} l!} \hspace*{1pt} (1 - t^{2})^{m/2} \frac{\mathrm{d}^{l+m}}{\mathrm{d}t^{l+m}} \hspace*{1pt} (t^{2\!} - 1)^{l\!}. 
\end{equation*}
\end{definition}

The associated Legendre polynomials satisfy a \emph{three-term recurrence relation} \cite[Eq.\ 8.5.3]{abramowitz}:
\begin{align}\label{eq:legendre_recurrence}
(l + 1 - m) \hspace*{1pt} P_{l+1,m}(t) \hspace*{2pt}=\hspace*{4pt} &(2l + 1) \hspace*{1pt} t \hspace*{1pt} P_{lm}(t) \notag\\&-\hspace*{2pt} (l + m) \hspace*{1pt} P_{l-1,m}(t), ~~~~~ t \in [-1,1], ~~~ |m| \leq l \in \mathbb{N}.
\end{align}

\begin{sloppypar}
In our context, the most important properties of the spherical harmonics are the following; for a detailed introduction to the 
related theory, refer to \citep{freeden} or \citep{dai_xu}, for example.
\end{sloppypar}

\begin{theorem}\label{thm:spherical_harmonics}
The spherical harmonics constitute an orthonormal basis of the space $L^{2}(\mathbb{S}^{2})$ of square-integrable functions on 
the unit sphere \hspace*{1pt}$\mathbb{S}^{2\!}$, endowed with the standard inner product
\begin{equation}\label{eq:S2_inner_product}
\langle f, g \rangle_{\mathbb{S}^{2}} \hspace*{2pt}\coloneqq\hspace*{2pt}\! \int_{0}^{\pi} \! \int_{0}^{2\pi}\!\! f(\vartheta,\varphi) \hspace*{1pt} \overline{g(\vartheta,\varphi)} \hspace*{2pt} \mathrm{d}\varphi \hspace*{1pt} \sin\vartheta \hspace*{2pt} \mathrm{d}\vartheta, ~~~~~~ f,g \in L^{2}(\mathbb{S}^{2}).
\end{equation}
Furthermore, the spherical harmonics of degree at most $N$\!\hspace*{0.5pt} span the space of all (complex-valued) polynomials of 
(total) degree at most $N$\!\hspace*{0.5pt} on \hspace*{1pt}$\mathbb{S}^{2\!}$ ($N\! \in \mathbb{N}_{0}$).
\end{theorem}

With this knowledge, it is clear that the spherical harmonics are a good choice for the spherical part $S$ in \eqref{eq:product_ansatz}; 
the orthogonality relation \eqref{eq:O2} is thus satisfied (compare with \eqref{eq:S2_inner_product}).

In a next step, the spherical harmonics are extended radially in order to 
regain polynomials on $\mathbb{R}^{3\!}$. To this end, \citeauthor{gto} borrow the following result from the theory of orthogonal 
polynomials in the univariate setting:

\begin{theorem}[\textbf{\cite[Sect.\ 5.1]{szego}}]\label{thm:generalized_laguerre}
For every fixed real number $\alpha > -1$, there exists exactly one set of polynomials on the positive half-line $[0,\infty)$ 
satisfying the orthogonality relation
\begin{equation}\label{eq:laguerre_ortho}
\int_{0}^{\infty} \!\! L_{j}^{(\alpha)\!}(t) \hspace*{1pt} L_{k}^{(\alpha)\!}(t) \hspace*{1.5pt} t^{\alpha} \hspace*{0pt} \mathrm{e}^{-t} \hspace*{1pt} \mathrm{d}t
\hspace*{2pt}=\hspace*{2pt}
\frac{\Gamma(k + \alpha + 1)}{k!} \hspace*{1pt} \delta_{jk}, ~~~~~~ j,k \in \mathbb{N}_{0},
\end{equation}
where $\Gamma\!$ denotes the \emph{gamma function}. 
\end{theorem}

These polynomials are called \emph{generalized} (or \emph{associated}) \emph{Laguerre polynomials}. Each generalized Laguerre 
polynomial $L_{k}^{(\alpha)\!}$ is of degree $k$, and possesses the closed-form expression
\begin{equation}\label{eq:associated_laguerre_explicit}
L_{k}^{(\alpha)\!}(t) \hspace*{2pt}=\hspace*{2pt} \sum_{j=0}^{k} \frac{(-1)^{j}}{j!} {k + \alpha \choose k - j} \hspace*{1pt} t^{\hspace*{0.5pt}j\!}, ~~~~~~ t \in [0,\infty)
\end{equation}
\cite[Eq.\ 5.1.6]{szego}. As the associated Legendre polynomials, the generalized Laguerre polynomials satisfy a three-term 
recurrence relation \cite[Eq.\ 5.1.10]{szego}:
\begin{align}\label{eq:laguerre_recursion}
(k + 1) \hspace*{1pt} L_{k+1}^{(\alpha)}\!\hspace*{1pt}(t) 
\hspace*{2pt}=\hspace*{4pt}
&(2k + \alpha + 1 - t) \hspace*{1pt} L_{k}^{(\alpha)}\!\hspace*{1pt}(t) \notag\\ &-\hspace*{2pt} (k + \alpha) \hspace*{1pt} L_{k-1}^{(\alpha)}\!\hspace*{1pt}(t), ~~~~~~~~ t \in [0,\infty), ~~~ k \in \mathbb{N}.
\end{align}

Inspired by the solution to Schr\"{o}dinger's equation for the hydrogen atom (cf.\ \cite[Sect.\ 7.4]{hydrogen}), 
\citeauthor{gto} now make the ansatz
\begin{equation*}
R(r) \hspace*{2pt}=\hspace*{2pt} R_{k}^{(\alpha)\!}(r) \hspace*{2pt}\coloneqq\hspace*{2pt} N_{k}^{(\alpha)\!} r^{\alpha} L_{k}^{(\alpha + 1/2)\!}(r^{2}), ~~~ \alpha > -1, ~~~ k \in \mathbb{N}_{0}, 
\end{equation*}
for the radial part $R$ in \eqref{eq:product_ansatz}. By setting $\alpha \coloneqq l$, where $l$ is the order of the spherical harmonic 
$Y_{lm}$ to be extended, 
and substituting $r^{2\!}$\hspace*{1pt} for $t$ in \eqref{eq:laguerre_ortho}, this 
ansatz results in the orthogonality relation
\begin{equation}\label{eq:laguerre_ortho_transformed}
\int_{0}^{\infty} \!\! R_{j}^{(l)\!}(r) \hspace*{1pt} R_{k}^{(l)\!}(r) \hspace*{1.5pt} r^{2} \hspace*{1pt} \mathrm{e}^{-r^{2\!}} \hspace*{1pt} \mathrm{d}r
\hspace*{2pt}=\hspace*{2pt}
\frac{\Gamma(k + l + 3/2)}{2 k!} \hspace*{1pt} \Bigl\lbrace N_{k}^{(l)\!}\Bigr\rbrace^{\!2\!} \hspace*{1pt} \delta_{jk}, ~~~~~~ j,k \in \mathbb{N}_{0}.
\end{equation}
This immediately entails setting
\begin{equation*}
N_{k}^{(l)} \hspace*{2pt}\coloneqq\hspace*{2pt} \sqrt{\frac{2k!}{\Gamma(k + l + 3/2)}},
\end{equation*}
so that the orthogonality relation \eqref{eq:O1} is satisfied. Observe that the polynomials $R_{k}^{(l)\!}$
are real.

At this point, it is important to note that one can not expect to obtain a polynomial on $\mathbb{R}^{3\!}$ by extending a spherical 
harmonic $Y_{lm}$ by an arbitrary polynomial in $r$. However, the above combinations of $R_{k}^{(l)\!}$ and $Y_{lm}$ are, in fact, 
polynomials of (total) degree $2k + l$ on $\mathbb{R}^{3\!}$. This is due to the fact that $r^{l} Y_{lm}$ is a polynomial of degree 
$l$ on $\mathbb{R}^{3\!}$, while $L_{k}^{(l + 1/2)\!}(r^{2})$ is a polynomial of degree $2k$ on $\mathbb{R}^{3\!}$. By some further 
working with the closed-form expression \eqref{eq:associated_laguerre_explicit} of the generalized Laguerre polynomials, \citeauthor{gto} 
found that by setting $k \coloneqq n - l - 1$, $n > l$, the arising combinations of $R_{nl} \coloneqq R_{n-l-1\!}^{(l)}$ and $Y_{lm}$ 
actually span the space of polynomials on $\mathbb{R}^{3\!}$. This establishes the final form of the SGL basis functions $H_{nlm}$ of 
Definition \ref{def:sgl}. The notion `basis functions' is justified by Theorem \ref{thm:density}.

Finally, note that the degree of the SGL basis functions is not to be confused with their polynomial degree: The SGL basis function 
$H_{nlm}$ is of degree $n$ in the sense of Definition \ref{def:sgl}, but of polynomial degree $2n - l - 2$. 

\section{Fast Fourier transforms for SGL basis functions}\label{sec:sglft}
In this section, we develop fast Fourier transforms for the SGL basis functions of Definition \ref{def:sgl}. 
To this end, we first derive an SGL sampling theorem for \emph{bandlimited functions}. For a fixed 
\emph{bandlimit} $B \!\hspace*{1pt}\in \mathbb{N}$, these are functions $f \in H$ for which $\hat{f}_{nlm\!} 
= 0$ if $n > B$. By construction of the SGL basis functions, with increasing bandlimit $B$, these  
spaces exhaust the entire class of polynomials on $\mathbb{R}^{3}$; recall, however, that these spaces do not 
coincide with the classical polynomial spaces of $\mathbb{R}^{3\!}$. This is why we introduce a different 
notion here. The SGL sampling theorem enables us to compute the SGL Fourier coefficients of such bandlimited 
functions in a \emph{discrete} way, that is, with a finite number of computation steps. This immediately results 
in a first discrete SGL Fourier transform and corresponding inverse. By a separation-of-variables 
technique and the employment of a fast spherical Fourier transform, we then unveil a whole class of fast SGL 
Fourier transforms and inverses. We close this section by an linear-algebraic description and comparison 
of our transforms.

\subsection{SGL sampling theorem}\label{subsec:SGL_sampling_theorem}

To derive an SGL sampling theorem, we make use of two auxiliary results: an \emph{equiangular quadrature 
rule} for the unit sphere $\mathbb{S}^{2\!}$, which is a classical construct of \cite{sfft}, and a \emph{Gauss-Hermite 
quadrature rule} for the positive half line $[0,\infty)$. We begin with the former.

\begin{theorem}[\textbf{\citet[Theorem 3]{sfft}}]\label{thm:S2_sampling_theorem}
Let $g$ be a polynomial of degree $L-1$ on \hspace*{1pt}$\mathbb{S}^{2\!}$, i.e., 
$g \in \textnormal{span}\lbrace Y_{lm\!} : |m| \leq l < L \rbrace$, $L \in \mathbb{N}$.
Then the spherical Fourier coefficients of $g$ obey the quadrature rule
\begin{equation}\label{eq:S2_sampling_theorem}
\langle g, Y_{lm} \rangle_{\mathbb{S}^{2}} \hspace*{2pt}=\hspace*{2pt}\! \sum_{j,k = 0}^{2L - 1} b_{\!j} \hspace*{1pt} g(\vartheta_{j}, \varphi_{k}) \hspace*{1pt} \overline{Y_{lm}(\vartheta_{j}, \varphi_{k})}, ~~~~~~ |m| \leq l < L,
\end{equation}
where the sampling angles are defined as $\vartheta_{j} \coloneqq (2j + 1)\pi / 4L$ and $\varphi_{k} \coloneqq k\pi / L$, 
resulting in the closed-form expression
\begin{equation*}
b_{\!j} \hspace*{2pt}=\hspace*{2pt} \sin\Bigl(\!(2j + 1) \frac{\pi}{4L}\Bigr) \frac{2}{L} \sum_{l = 0}^{L - 1} \frac{1}{2l + 1} \hspace*{0pt} \sin\Bigl(\!(2j + 1)(2l + 1) \frac{\pi}{4L}\Bigr)
\end{equation*}
for the quadrature weights.
\end{theorem}

\begin{figure}[t]
 \centering
 \includegraphics[width=\textwidth]{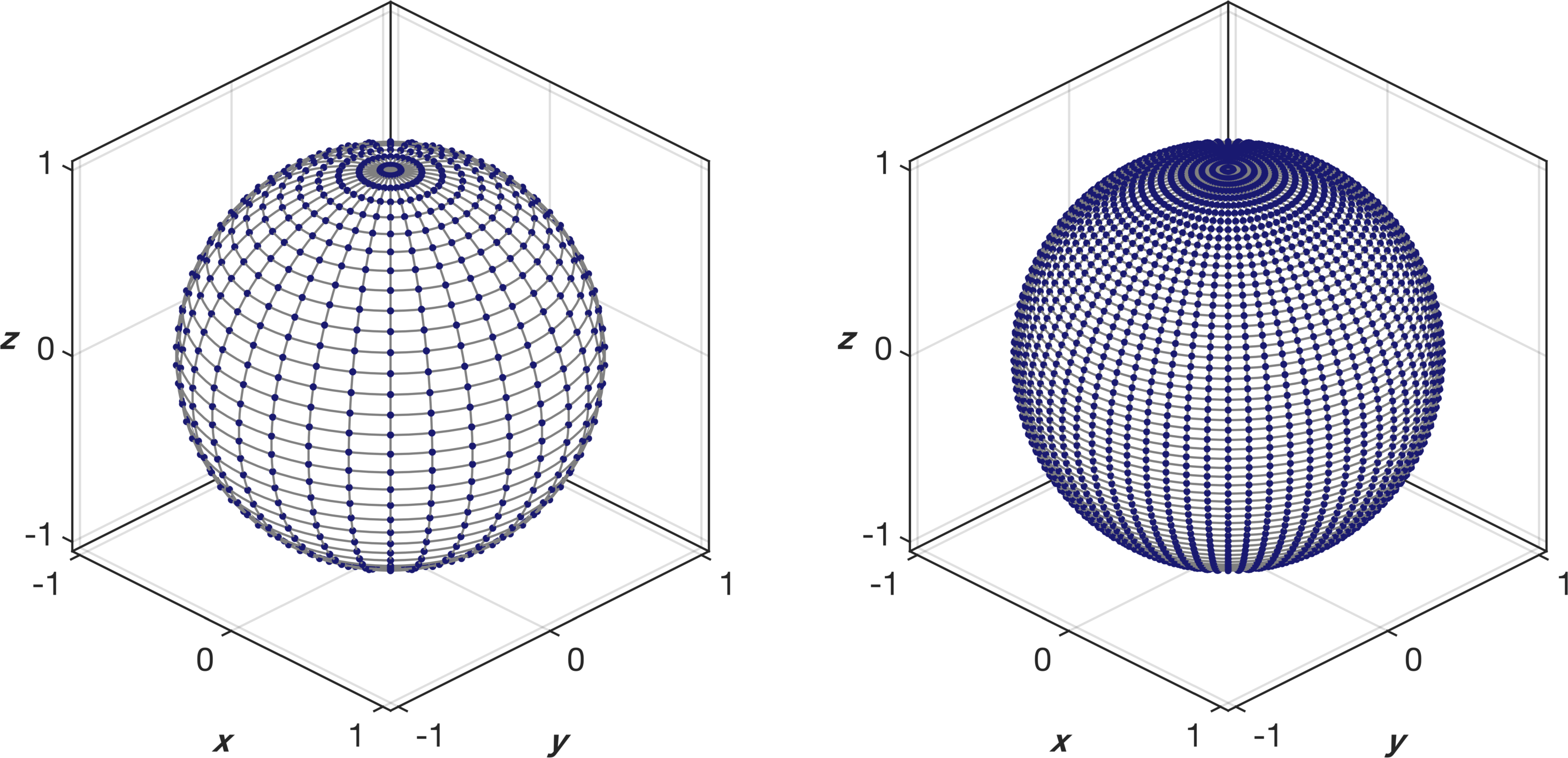}\vspace*{5pt}
 \caption{Sampling angles $(\vartheta_{j}, \varphi_{k})$, plotted as points on the unit sphere $\mathbb{S}^{2\!}$, for 
 \textbf{(left)} $L = 16$ and \textbf{(right)} $L = 32$. Note that the sampling angles 
 are denser near the poles than near the equator.}
 \label{fig:S2_sampling_angles}
\end{figure}

\begin{figure}[t]
 \centering
 \includegraphics[width=\textwidth]{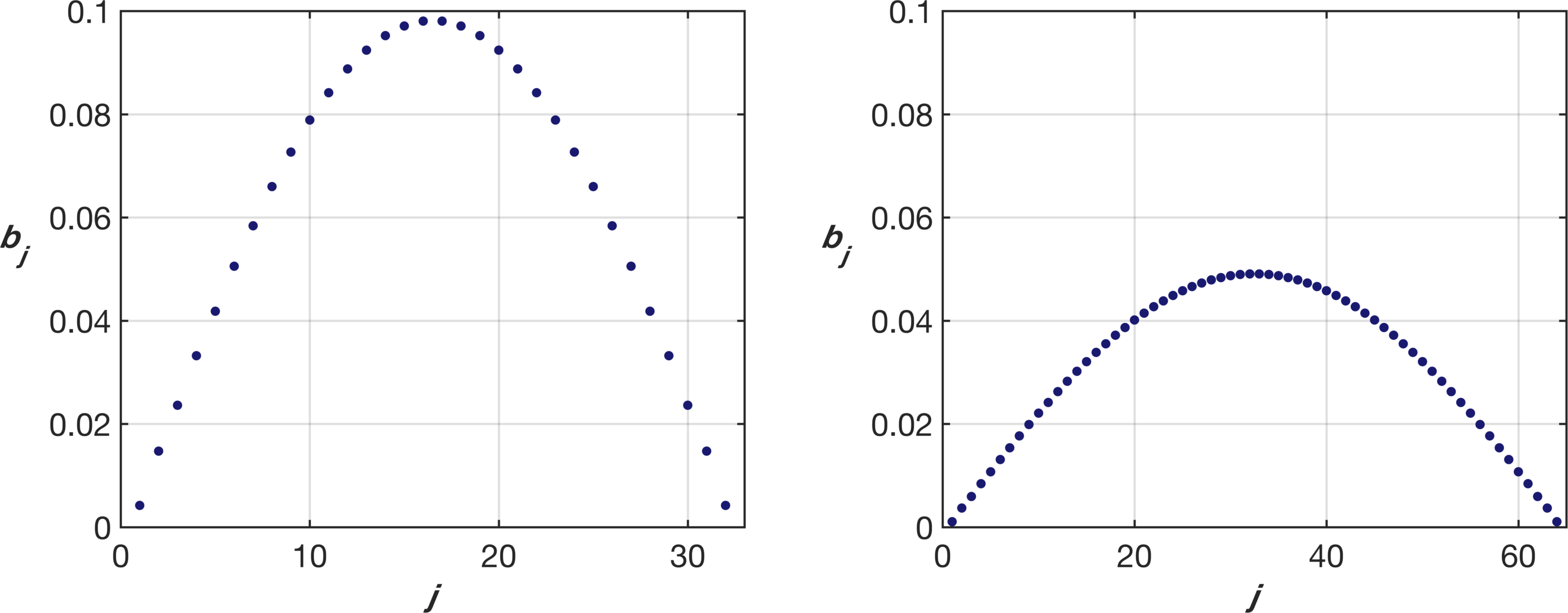}\vspace*{3pt}
 \caption{Spherical quadrature weights $b_{\!j}$, plotted for \textbf{(left)} $L = 16$ and 
 \textbf{(right)} $L = 32$. Note how the weights compensate for the higher density of sampling 
 angles near the poles of \hspace*{1pt}$\mathbb{S}^{2\!}$ (cf.\ Figure \ref{fig:S2_sampling_angles}):\ the higher the density 
 of sampling angles gets, the smaller the corresponding weights become.}
 \label{fig:S2_weights}
\end{figure}

We call $L$ the \emph{order} of the respective spherical quadrature rule.
Note that the quadrature weights $b_{\!j}$ are real, and do not depend on the azimuthal sampling angles $\varphi_{k}$. 
This is due to the special choice of the sampling angles $\vartheta_{j}$ and $\varphi_{k}$. Figure \ref{fig:S2_sampling_angles} 
shows the sampling angles $\vartheta_{j}$ and $\varphi_{k}$ for the orders $L = 16$ and $L = 32$, respectively, plotted 
as points on the unit sphere $\mathbb{S}^{2\!}$. Figure \ref{fig:S2_weights} shows the corresponding quadrature weights 
$b_{\!j}$. 

As it turns out, the weights $b_{\!j}$ are \emph{positive}. Since we are not aware of a proof of this 
feature having been given in this context, we include a direct proof here.

\begin{lemma}\label{thm:gewichte}
The quadrature weights $b_{\!j}$ are positive.
\end{lemma}

\begin{proof}
Let $L \in \mathbb{N}$ be given. Firstly, we note that $0 < (2j + 1) \pi / 4L < \pi$ and thus $0 < \sin((2j + 1) \pi 
/ 4L)$ for $j = 0,\dots,2L - 1$. Set $\gamma_{j} \coloneqq (2j + 1) \pi / 4$, $j \in \lbrace 0,\dots,2L - 1 \rbrace$. 
We derive
\begin{align}
\sum_{l = 0}^{L - 1} \frac{1}{2l + 1} \sin\Bigl(\!(2l + 1) \frac{\gamma_{j}}{L}\Bigr)\!\notag
\hspace*{2pt}&=\hspace*{2pt} \mathfrak{Im} \sum_{l = 0}^{L - 1} \frac{1}{2l + 1} \hspace*{1.5pt} \mathrm{e}^{\mathrm{i}(2l + 1) \gamma_{j} / L}\notag\\
\hspace*{2pt}&=\hspace*{2pt} \mathfrak{Im} \sum_{l = 0}^{L - 1} \hspace*{1pt} \Bigl\lbrace \frac{\mathrm{i}}{L}\! \int_{0}^{\gamma_{j}}\! \mathrm{e}^{\mathrm{i} (2l + 1) t / L} \hspace*{1pt} \mathrm{d}t \hspace*{1.5pt}+\hspace*{0.5pt} \frac{1}{2l + 1} \Bigr\rbrace\notag\\
\hspace*{2pt}&=\hspace*{2pt} \frac{1}{L} \hspace*{2pt} \mathfrak{Re}\! \int_{0}^{\gamma_{j}}\! \mathrm{e}^{\mathrm{i} t / L} \sum_{l = 0}^{L - 1} \mathrm{e}^{\mathrm{i} 2 l t / L} \hspace*{1pt} \mathrm{d}t\notag\\
\hspace*{2pt}&=\hspace*{2pt} \frac{1}{L} \hspace*{2pt} \mathfrak{Re}\! \int_{0}^{\gamma_{j}}\!\! \frac{\mathrm{e}^{2 \mathrm{i} t} - 1}{\mathrm{e}^{\mathrm{i} t / L\!} - \mathrm{e}^{- \mathrm{i} t / L}} \hspace*{1pt} \mathrm{d}t\vphantom{\sum_{l = 0}^{L - 1}}\notag\\
\hspace*{2pt}&=\hspace*{2pt} \frac{1}{2L} \int_{0}^{\gamma_{j}}\!\! \frac{\sin(2t)}{\sin(t/L)} \hspace*{1pt} \mathrm{d}t.\vphantom{\sum^{L - 1}_{l=0}}\label{eq:last_step}
\end{align}

Substituting $u/2$ for $t$ on the right-hand side of \eqref{eq:last_step}, we arrive at
\begin{equation}\label{eq:integraldarstellung}
\sum_{l = 0}^{L - 1} \frac{1}{2l + 1} \sin\Bigl(\!(2j + 1)(2l + 1) \frac{\pi}{4L}\Bigr) \hspace*{2pt}=\hspace*{2pt} \frac{1}{4L} \int_{0}^{(j + 1/2)\pi}\!\!\!\! \frac{\sin u}{\sin(u/2L)} \hspace*{1pt} \mathrm{d}u.
\end{equation}

To show the positivity of the right-hand side of \eqref{eq:integraldarstellung}, we distinguish between
four different cases: $j < L$ or $j \geq L$, $j$ being even or uneven, respectively. The reason for the first distinction 
is that the denominator $\sin(\cdot/2L)$ is strictly increasing on the interval $[0,L\pi)$ and strictly decreasing on 
the interval $(L\pi,2L\pi]$. Furthermore, $\sin(\cdot/2L)$ is non-negative on the integration range $[0,(j+1/2)\pi] \subset 
[0,2L\pi]$ which allows all cases to be treated in a straightforward manner.

Let now $j < L$ and set $\kappa_{j} \coloneqq 0$ if $j$ is even and $\kappa_{j} \coloneqq 1$ if $j$ is odd. Two 
simple estimations reveal
\begin{align*}
\int_{0}^{(j + 1/2)\pi}\!\!\!\! \frac{\sin u}{\sin(u/2L)} \hspace*{1pt} \mathrm{d}u
\hspace*{2pt}&>\hspace*{2pt} \!\!\int_{0}^{(j + \kappa_{j})\pi}\!\!\!\! \frac{\sin u}{\sin(u/2L)} \hspace*{1pt} \mathrm{d}u \vphantom{\sum_{k = 0}}\\
\hspace*{2pt}&>\hspace*{2pt} \hspace*{-12pt}\sum_{k = 0}^{(j+\kappa_{j})/2 - 1} \hspace*{-3pt}\!\frac{1}{\sin((2k+1)\pi/2L)} \hspace*{1pt} \Bigl\lbrace \int_{2k\pi}^{(2k+1)\pi}\!\!\!  + \int_{(2k+1)\pi}^{2(k+1)\pi}\Bigr\rbrace \sin u \hspace*{2.5pt} \mathrm{d}u\\[4pt] 
\hspace*{2pt}&=\hspace*{2pt} 0.
\end{align*}
If, on the other hand, $j \geq L$, we make use of the identity
\begin{equation*}
\int_{0}^{(j + 1/2)\pi}\!\!\!\! \frac{\sin u}{\sin(u/2L)} \hspace*{1pt} \mathrm{d}u
\hspace*{2pt}=\hspace*{2pt} - \!\int_{0}^{(2L - j-1/2)\pi}\!\!\!\! \frac{\sin u}{\sin(u/2L)} \hspace*{1pt} \mathrm{d}u
\end{equation*}
and proceed in the same sense. \hfill $\square$
\end{proof}

In a next step towards our SGL sampling theorem, we introduce the \emph{half-range Gauss-Hermite quadrature}, i.e., a 
Gaussian quadrature rule for the Hermite weight $\exp(-r^{2})$ on the positive half line $[0, \infty)$. We 
add the term `half-range' here because the Hermite weight is usually considered on the entire real line 
$\mathbb{R}$, leading to other quadrature rules.

\begin{theorem}[\textbf{\cite[Sects. 3.2.2 \& 3.2.3]{numerical_analysis}}]\label{thm:half-range_Hermite}
Let $p$ be a polynomial of degree at most \hspace*{1pt}$2N - 1$, $N \!\in \mathbb{N}$. Furthermore, let 
\hspace*{0.25pt}$r_{0} \!< \!\cdots\! <\! r_{N - 1}$ denote the simple, positive roots of the $N$th orthogonal polynomial $p_{N}$ 
with respect to the weight function $\exp(-r^{2})$\! on $[0,\infty)$. Then equality holds in the Gaussian 
quadrature formula
\begin{equation}\label{eq:half-range_Hermite}
\int_{0}^{\infty}\!\! p(r) \hspace*{1pt} \mathrm{e}^{-r^{2\!}} \hspace*{1pt} \mathrm{d}r \hspace*{2pt}=\hspace*{2pt}\! \sum_{i = 0}^{N - 1} a_{i} \hspace*{1.5pt} p(r_{i}),
\end{equation}
where the quadrature weights $a_{i}$ are real, positive, and satisfy the equation
\begin{equation*}
a_{i} \hspace*{2pt}=\hspace*{2pt}\! \int_{0}^{\infty}\!\!\! \frac{p_{N}(r)}{(r - r_{i}) \hspace*{1pt} p'_{N}(r_{i})} \hspace*{1.5pt} \mathrm{e}^{-r^{2\!}} \hspace*{1pt} \mathrm{d}r.
\end{equation*}
\end{theorem}

As in the spherical quadrature rules introduced in Theorem \ref{thm:S2_sampling_theorem}, we call $N$ the \emph{order} 
of the respective quadrature rule.

We do not want to go into detail regarding the numerical aspects of Theorem \ref{thm:half-range_Hermite}. 
We only mention that \citet[Sect.\ 2]{half_range_hermite} have developed special recurrence relations to compute 
the coefficients of the three-term recurrence relation satisfied by the orthogonal polynomials $p_{n}$, $n \in 
\mathbb{N}_{0}$. This, in turn, allows the sampling points $r_{i}$ in \eqref{eq:half-range_Hermite} to be computed by 
a standard approach (see \cite[Sect.\ 3.2.2, (v)]{numerical_analysis}). It is then also possible to compute the 
corresponding quadrature weights $a_{i}$ with desired precision by \cite[Eq.\ 2.1]{half_range_hermite}. This 
approach is used in the numerical experiments of the upcoming Section \ref{sec:experiments}. Figure 
\ref{fig:half_range_hermite} shows the sampling points $r_{i}$ and corresponding weights $a_{i}$ for 
the orders $N\! = 32$ and $N\! = 64$.

\begin{figure}[t]
 \centering
 \includegraphics[width=\textwidth]{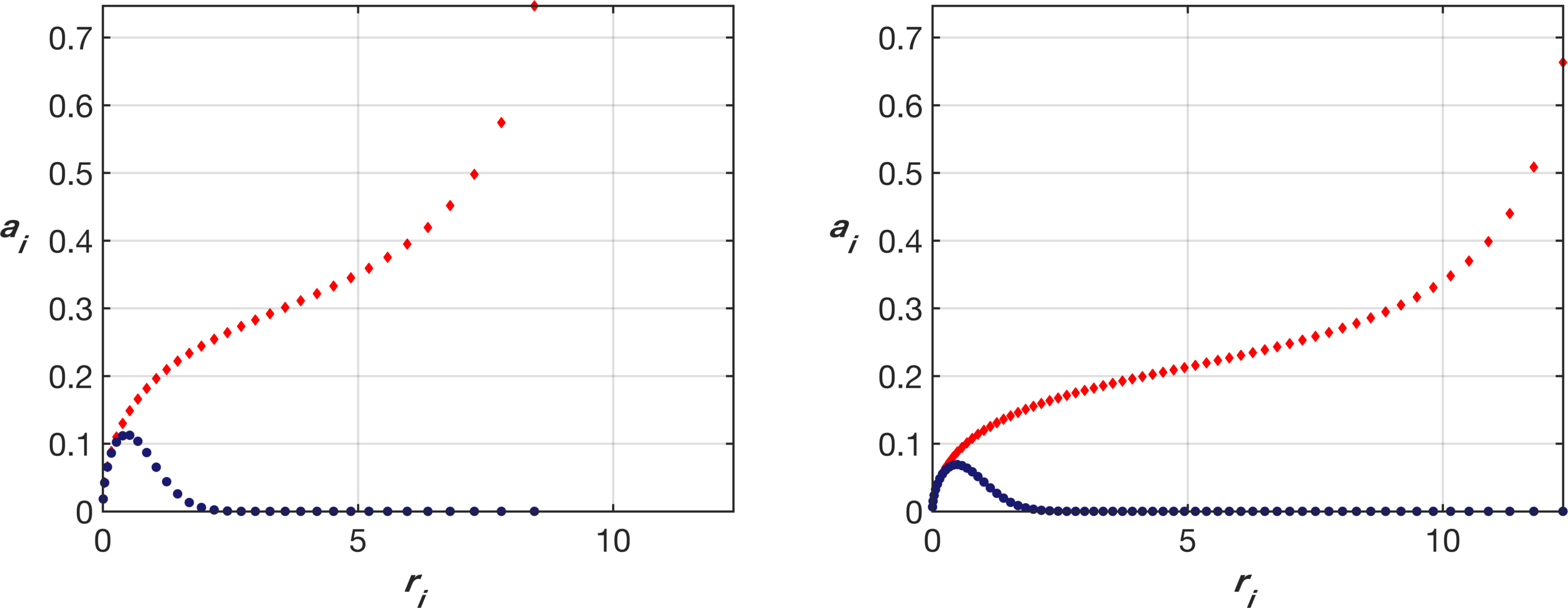}\vspace*{-2pt}
 \caption{Sampling points $r_{i}$, corresponding quadrature weights $a_{i\!}$ (\color{dark_blue}$\bullet$\color{black}), 
 and scaled weights $a_{i} \exp(r_{i}^{2})$ (\color{red}$\blacklozenge$\color{black}, cf.\ Section \ref{subsec:DSGLFT}) 
 of the half-range Gauss-Hermite quadrature rule of order \textbf{(left)} $N\! = 32$ and \textbf{(right)} $N\! = 64$. 
 The points $r_{i}$ are used as sampling radii which, combined with the sampling angles $\vartheta_{j}$ and $\varphi_{k}$ 
 shown in Figure \ref{fig:S2_sampling_angles}, constitute the sampling points of our SGL sampling theorem (Theorem 
 \ref{thm:SGL_sampling_theorem}) for the bandlimits (left) $B = 16$ and (right) $B = 32$.}\vspace*{-4pt}
 \label{fig:half_range_hermite}
\end{figure}

We now combine Theorems \ref{thm:S2_sampling_theorem} and \ref{thm:half-range_Hermite} to obtain our SGL 
sampling theorem. For this, let $f$ be bandlimited with bandlimit $B \in \mathbb{N}$. The function $f$ thus
possesses the unique SGL decomposition
\begin{equation}\label{eq:SGL_decomposition}
f \hspace*{2pt}=\hspace*{2.5pt}\! \sum_{n = 1}^{B} \hspace*{3.5pt} \sum_{l = 0}^{n - 1} \sum_{m = -l}^{l}\! \hat{f}_{nlm} \hspace*{1.5pt} H_{nlm}.
\end{equation}

Recalling that $H_{nlm}(r,\vartheta,\varphi) = N_{nl} \hspace*{1pt} R_{nl}(r) Y_{lm}(\vartheta,\varphi)$, we see 
that $f(r,\cdot,\cdot)$ is a linear combination of spherical harmonics of degree $l < B$ for every fixed $r \in 
[0,\infty)$. Hence, using the spherical quadrature rule of Theorem \ref{thm:S2_sampling_theorem} of order $L = B$, 
we get for $|m| \leq l < n \leq B$
\begin{align}
\hat{f}_{nlm} 
\hspace*{2pt}&=\hspace*{2pt} N_{nl\!} \int_{0}^{\infty}\! \Bigl\lbrace\hspace*{0pt}\int_{0}^{\pi} \! \int_{0}^{2\pi}\!\! f(r,\vartheta,\varphi) \hspace*{1pt} \overline{Y_{lm}(\vartheta,\varphi)} \hspace*{2pt} \mathrm{d}\varphi \hspace*{1pt} \sin\vartheta \hspace*{2pt} \mathrm{d}\vartheta \Bigr\rbrace R_{nl}(r) \hspace*{1pt} r^{2} \hspace*{1pt} \mathrm{e}^{-r^{2\!}} \hspace*{1pt} \mathrm{d}r\notag\\
\hspace*{2pt}&=\hspace*{2pt} N_{nl\!} \int_{0}^{\infty}\! \Bigl\lbrace\hspace*{0pt}\sum_{j,k = 0}^{2B - 1} b_{\!j} \hspace*{1pt} f(r,\vartheta_{j},\varphi_{k}) \hspace*{1pt} \overline{Y_{lm}(\vartheta_{j},\varphi_{k})} \Bigr\rbrace R_{nl}(r) \hspace*{1pt} r^{2} \hspace*{1pt} \mathrm{e}^{-r^{2\!}} \hspace*{1pt} \mathrm{d}r.\label{eq:zwischenstufe}
\end{align}

Considering again the SGL decomposition \eqref{eq:SGL_decomposition}, we verify that the integrand 
in \eqref{eq:zwischenstufe} is a polynomial in $r$ of degree at most $4B - 2$, multiplied by the Hermite weight. 
Therefore, using the half-range Gauss-Hermite quadrature rule of Theorem \ref{thm:half-range_Hermite} of order 
$N\! = 2B$, we obtain
\begin{align}
\int_{0}^{\infty}\! \Bigl\lbrace\hspace*{0pt}\sum_{j,k = 0}^{2B - 1}\! b_{\!j} \hspace*{1pt} &f(r,\vartheta_{j},\varphi_{k}) \hspace*{1pt} \overline{Y_{lm}(\vartheta_{j},\varphi_{k})} \Bigr\rbrace R_{nl}(r) \hspace*{1pt} r^{2} \hspace*{1pt} \mathrm{e}^{-r^{2\!}} \hspace*{0pt} \mathrm{d}r \hspace*{2pt}=\hspace*{2pt}\notag\\
\!&\sum_{i,j,k = 0}^{2B - 1}\! a_{i} \hspace*{1pt} r_{i}^{2} \hspace*{1pt} b_{\!j} \hspace*{1pt} f(r_{i},\vartheta_{j},\varphi_{k}) \hspace*{0pt} R_{nl}(r_{i}) \hspace*{0pt} \overline{Y_{lm}(\vartheta_{j},\varphi_{k})}.\label{eq:zwischenstufe_2}
\end{align}

\begin{sloppypar}
Combining \eqref{eq:zwischenstufe} and \eqref{eq:zwischenstufe_2} 
now yields our SGL sampling theorem:
\end{sloppypar}

\begin{theorem}[\textbf{SGL sampling theorem}]\label{thm:SGL_sampling_theorem}
Let $f$ be a bandlimited function with bandlimit $B \in \mathbb{N}$. Then the SGL Fourier coefficients of $f$ obey the 
quadrature rule
\begin{equation}\label{eq:SGL-Sampling-Theorem}
\hat{f}_{nlm} \hspace*{2pt}=\hspace*{2pt}\!\! \sum_{i,j,k = 0}^{2B - 1}\! a_{i} \hspace*{1pt} r_{i}^{2} \hspace*{1pt} b_{\!j} \hspace*{1pt} f(r_{i},\vartheta_{j},\varphi_{k}) \hspace*{1pt} \overline{H_{nlm}(r_{i}, \vartheta_{j},\varphi_{k})}, ~~~~~~|m| \leq l < n \leq B,
\end{equation}
where the sampling radii $r_{i}>0$ and weights $a_{i}>0$ are those of the half-range Gauss-Hermite quadrature 
rule of order $2B$ (Theorem \ref{thm:half-range_Hermite}), while the sampling angles $(\vartheta_{j}, \varphi_{k})$ 
and weights $b_{\!j}>0$ are those of the equiangular spherical quadrature rule of order $B$ (Theorem 
\ref{thm:S2_sampling_theorem}, Lemma \ref{thm:gewichte}).
\end{theorem}

Note that the sampling angles shown in Figure \ref{fig:S2_sampling_angles} are radially extended by precisely the sampling 
points shown in Figure \ref{fig:half_range_hermite} to obtain the sampling points of Theorem \ref{thm:SGL_sampling_theorem}
for the bandlimits $B = 16$ and $B = 32$, respectively.

\subsection{Discrete SGL Fourier transforms}\label{subsec:DSGLFT}

Based on the results of the previous section, we are now able to give a rigorous definition of the term `discrete SGL 
Fourier transform'.

\begin{definition}[\textbf{DSGLFT/iDSGLFT}]
Let $B\! \in \mathbb{N}$. Any method for the computation of the SGL Fourier coefficients of bandlimited functions with 
bandlimit $B$ by means of \eqref{eq:SGL-Sampling-Theorem} is called a \emph{discrete SGL Fourier transform} 
(DSGLFT). Correspondingly, any method for reconstruction of function values of functions with bandlimited 
$B$ at the respective sampling nodes $(r_{i},\vartheta_{j},\varphi_{k})$ is referred to as an \emph{inverse 
discrete SGL Fourier transform} (iDSGLFT).
\end{definition}

Let $\tilde{a}_{i} \coloneqq a_{i} \exp(r_{i}^{2}) \hspace*{1pt} r_{i}^{2}$.
We state a simple DSGLFT as Algorithm \ref{alg:DSGLFT}. The SGL Fourier coefficients of a bandlimited 
function $f$ are here computed one after another, evaluating the corresponding triple sum every 
single time. We introduce the factor $\exp(r_{i}^{2})$ to compensate 
for the fast decay of the quadrature weights $a_{i}$ (cf.\ Figure \ref{fig:half_range_hermite}). This modification 
is accounted for by weighting the SGL basis function samples $H_{nlm}(r_{i},\vartheta_{j},\varphi_{k})$ by 
the factor $\exp(-r_{i}^{2})$ (see also Section \ref{sec:discussion}).

In this work, we use the standard complexity model in which a single operation is defined as a complex 
multiplication and a subsequent complex addition.
To state the asymptotic complexity of Algorithm \ref{alg:DSGLFT}, we make the assumption that the (modified) 
quadrature weights $\tilde{a}_{i\!}$ and $b_{\!j}$, as well as the sampling 
points $(r_{i},\vartheta_{j},\varphi_{k})$, are stored and readily available during runtime. Using the three-term 
recurrence relations \eqref{eq:legendre_recurrence} and \eqref{eq:laguerre_recursion} of the associated Legendre 
polynomials and the generalized Laguerre polynomials, we can evaluate any SGL basis function 
$H_{nlm}$ at an arbitrary sampling node $(r_{i},\vartheta_{j},\varphi_{k})$ in $\mathcal{O}(B)$ steps. 
Algorithm \ref{alg:DSGLFT} has, thus, an asymptotic complexity of $\mathcal{O}(B^{7})$: the total number of summands 
of the triple sum scales with $B^{3\!}$, just as the total number of iterations of the three `for' loops.

\begin{algorithm}[t]\vspace*{4pt}
 \KwData{Sample values $f(r_{i}, \vartheta_{j}, \varphi_{k})$; $i,j,k = 0,\dots,2B\!-\!1$, of a function $f$ with bandlimit $B\!\hspace*{1pt} \in \mathbb{N}$}\vspace*{3pt}
 \KwResult{SGL Fourier coefficients $\hat{f}_{nlm}$, $|m| \leq l < n \leq B$}\vspace*{3pt}
 \For{$n = 1$ \KwTo $B$}
 {
    \For{$l = 0$ \KwTo $n - 1$}
    {
       \For{$m = -\hspace*{1pt}l$ \KwTo $l$}
       {
          \vspace*{3pt}
          Compute\vspace*{3pt}
          $\hat{f}_{nlm} \hspace*{2pt}=\hspace*{0pt} \sum\limits_{i = 0}^{2B - 1} \hspace*{1pt} \sum\limits_{j = 0}^{2B - 1} \hspace*{1pt} \sum\limits_{k = 0}^{2B - 1}\! \big\lbrace \tilde{a}_{i} \hspace*{1pt} b_{\!j} \hspace*{1pt} f(r_{i}, \vartheta_{j}, \varphi_{k}) \big\rbrace \hspace*{1pt} \big\lbrace \overline{H_{nlm}(r_{i}, \vartheta_{j}, \varphi_{k})} \hspace*{1pt} \mathrm{e}^{-r_{i}^{2}} \big\rbrace$;
       }
    }
 }
 \caption{Naive DSGLFT}\label{alg:DSGLFT}
\end{algorithm}

We state a simple iDSGLFT as Algorithm \ref{alg:iDSGLFT}. The function values of a bandlimited function $f$
are here reconstructed at each sampling node $(r_{i},\vartheta_{j},\varphi_{k})$ by directly summing up the SGL basis 
function values $H_{nlm}(r_{i},\vartheta_{j},\varphi_{k})$, weighted by the respective SGL Fourier coefficient $\hat{f}_{nlm}$. 
Simple considerations show that this algorithm also has an asymptotic complexity of $\mathcal{O}(B^{7})$. 

\begin{algorithm}[t]\vspace*{4pt}
 \KwData{SGL Fourier coefficients $\hat{f}_{nlm}$, $|m| \leq l < n \leq B$, of a function $f$ with bandlimit $B\!\hspace*{1pt} \in \mathbb{N}$}\vspace*{3pt}
 \KwResult{Function values $f(r_{i}, \vartheta_{j}, \varphi_{k})$; $i,j,k = 0,\dots,2B-1$}\vspace*{3pt}
 \For{$i = 0$ \KwTo $2B - 1$}
 {
    \For{$j = 0$ \KwTo $2B - 1$}
    {
       \For{$k = 0$ \KwTo $2B - 1$}
       {
          \vspace*{3pt}
          Compute\vspace*{3pt}
          $f(r_{i}, \vartheta_{j}, \varphi_{k}) \hspace*{2pt}=\hspace*{0pt} \sum\limits_{n = 1}^{B} \hspace*{3.5pt} \sum\limits_{l = 0}^{n - 1} \sum\limits_{m = -l}^{l} \hat{f}_{nlm} \hspace*{1pt} H_{nlm}(r_{i}, \vartheta_{j}, \varphi_{k})$;
       }
    }
 }
 \caption{Naive iDSGLFT} \label{alg:iDSGLFT}
\end{algorithm}

\subsection{Fast SGL Fourier transforms}\label{subsec:FSGLFT}

At this point, we are naturally faced with the task to develop discrete SGL Fourier transforms and corresponding 
inverse transforms with an asymptotic complexity of less than $\mathcal{O}(B^{7})$. This motivates:

\begin{definition}[\textbf{FSGLFT/iFSGLFT}]
We call any DSGLFT (iDSGLFT) with an asymptotic complexity of less than $\mathcal{O}(B^{7})$ a \emph{fast (inverse) 
SGL Fourier transform}, abbreviated FSGLFT (iFSGLFT, respectively). 
\end{definition}

In this section, we design such fast transforms and, simultaneously, corresponding fast inverse transforms in three 
main steps: 1) We separate the above naive DSGLFT/iDSGLFT (Algorithm \ref{alg:DSGLFT} and \ref{alg:iDSGLFT}, 
respectively) into a radial and a spherical subtransform. 2) Subsequently,
we employ a fast spherical Fourier transform and inverse to reduce the complexity of the spherical subtransform. 
3) We intruduce our new discrete $R$ transform, a tool to compute the collection of sums
\begin{equation}\label{eq:DRT}
N_{nl}\! \sum_{i=0}^{2B - 1}\! a_{i} \hspace*{1pt} r_{i}^{2} \hspace*{1pt} R_{nl}(r_{i}) \hspace*{1pt} s_{i}
\hspace*{2pt}=\hspace*{0pt}\!
\sum_{i=0}^{2B - 1}\! \big\lbrace N_{nl} \hspace*{1pt} R_{nl}(r_{i}) \hspace*{1pt} \mathrm{e}^{-r_{i}^{2}} \big\rbrace \hspace*{1pt} \big\lbrace \tilde{a}_{i} \hspace*{1pt} s_{i} \big\rbrace, \:~~~~ n = l + 1,\dots, B,
\end{equation}
for a fixed $0 \leq l < B$, $[s_{i}]_{i = 0,\dots,2B-1}$ being an input vector of length $2B$, to reduce the
complexity of the radial subtransform. For this purpose, we also intruduce a corresponding inverse discrete 
$R$ transform. In the following, we consistently use the notation $[a_{\nu}]_{\nu = 0,\dots,N - 1}$ to denote 
a (column) vector of length $N\!\hspace*{0pt} \in \mathbb{N}$ with (complex-valued) elements $a_{\nu}$.

Let $B \in \mathbb{N}$ and a function $f$ with bandlimit $B$ be given. In a first step, we rearrange the triple 
sum in Algorithm \ref{alg:DSGLFT} to obtain
\begin{equation}\label{eq:SGL-Sampling-Theorem_rearranged}
\hat{f}_{nlm} \hspace*{0pt}=\hspace*{0pt}\! \sum_{i=0}^{2B - 1}\!\! \big\lbrace N_{nl} \hspace*{1pt} R_{nl}(r_{i}) \hspace*{1pt} \mathrm{e}^{-r_{i}^{2}} \big\rbrace \big\lbrace \tilde{a}_{i\!}\! \sum_{j,k = 0}^{2B - 1}\! b_{\!j} \hspace*{1pt} f(r_{i},\vartheta_{j},\varphi_{k}) \hspace*{1pt} \overline{Y_{lm}(\vartheta_{j},\varphi_{k})} \big\rbrace, ~~~ |m| \leq l < n \leq B.
\end{equation}

Note that even without a fast algorithm here, the above separation of variables allows the complexity 
of Algorithms \ref{alg:DSGLFT} and \ref{alg:iDSGLFT} to be reduced to $\mathcal{O}(B^{6})$ by a simple
rearrangement of the computation steps: Precomputation of the inner sum in \eqref{eq:SGL-Sampling-Theorem_rearranged} 
for all $|m| \leq l < B$ and $i = 0,\dots,2B-1$ can be done in $\mathcal{O}(B^{6})$ steps. Subsequent 
evaluation of the outer sum for all $|m| \leq l < n \leq B$ can be done in $\mathcal{O}(B^{5})$ steps. 
The costs for evaluating $R_{nl}$ and $Y_{lm}$ are taken into account as $\mathcal{O}(B)$, respectively. 
The computation steps of the inverse transform may be rearranged in the same sense.
We maintain this strategy, and optimize the substeps.

Since $f$ is bandlimited with bandlimit $B$, we conclude that $f(r_{i},\cdot,\cdot)$ is a polynomial of degree
at most $B-1$ on $\mathbb{S}^{2\!}$ for each $i$. By Theorem \ref{thm:spherical_harmonics}, 
this implies that $f(r_{i},\cdot,\cdot)\!\hspace*{0.5pt} \in \textnormal{span}\lbrace Y_{lm\!} : |m| \leq l < B \rbrace$ 
(we have already made use of this fundamental feature of bandlimited functions in the derivation of the SGL 
sampling theorem, Theorem \ref{thm:SGL_sampling_theorem}). Therefore, by the spherical quadrature rule of order $B$ 
in Theorem \ref{thm:S2_sampling_theorem}, the inner sum in \eqref{eq:SGL-Sampling-Theorem_rearranged} equals the 
spherical Fourier coefficient $\langle f(r_{i},\cdot,\cdot), Y_{lm} \rangle_{\mathbb{S}^{2}}$. The computation 
of these inner sums thus amounts to the computation of all spherical Fourier coefficients of $f$ restricted to 
the sphere of radius $r_{i}$ for each $i$. 

The fast spherical Fourier transforms
described by \cite{sfft2} are a suitable means to solve this task. At the same time, the corresponding
fast inverse transforms allow the function values $f(r_{i},\vartheta_{j},\varphi_{k})$ to be reconstructed
from the spherical Fourier coefficients $\langle f(r_{i},\cdot,\cdot), Y_{lm} \rangle_{\mathbb{S}^{2}}$
for each sampling radius $r_{i}$. This constitutes the spherical part of our FSGLFTs and iFSGLFTs. We include a brief 
discussion on fast spherical Fourier transforms based on the spherical quadrature rule of Theorem 
\ref{thm:S2_sampling_theorem} in the upcoming Section \ref{subsubsec:SFFT}.

In order to compute the SGL Fourier coefficients $\hat{f}_{nlm}$ from the precomputed spherical Fourier 
coefficients $\langle f(r_{i},\cdot,\cdot), Y_{lm} \rangle_{\mathbb{S}^{2}}$, that is, to evaluate the 
outer sum in \eqref{eq:SGL-Sampling-Theorem_rearranged}, we use the above-mentioned discrete 
$R$ transform, running through all pairs of $m$ and $l$ with $|m| \leq l < B$ (cf.\ \eqref{eq:DRT} and 
\eqref{eq:SGL-Sampling-Theorem_rearranged}). This new transform is presented in the upcoming Section 
\ref{subsubsec:DRT}. The inverse discrete $R$ transform, also presented in Section \ref{subsubsec:DRT}, 
allows the spherical Fourier coefficients $\langle f(r_{i},\cdot,\cdot), Y_{lm} \rangle_{\mathbb{S}^{2}}$ 
to be reconstructed from the SGL Fourier coefficients $\hat{f}_{nlm}$ with the same asymptotic complexity
as the forward transform. The discrete $R$ transform and its inverse thus make up the radial part of our 
FSGLFTs and iFSGLFTs.

\subsubsection{Fast equiangular spherical Fourier transforms}\label{subsubsec:SFFT}

Let $g$ be a polynomial of degree $L\!-\!1$ on $\mathbb{S}^{2\!}$, i.e., $g \in \textnormal{span}\lbrace Y_{lm\!} : 
|m| \leq l < L \rbrace$, $L \in \mathbb{N}$. \emph{Fast equiangular spherical Fourier transforms} 
based on Theorem \ref{thm:S2_sampling_theorem} allow the spherical Fourier coefficients $\langle g, Y_{lm} 
\rangle_{\mathbb{S}^{2}}$ of $g$ to be computed with an asymptotic complexity of less than $\mathcal{O}(L^{5})$, 
which is associated with the naive approach (cf.\ \eqref{eq:S2_sampling_theorem}). A large class of such fast transforms 
was derived and thoroughly tested by \cite{sfft2}. In that work, the authors also presented corresponding 
fast inverse transforms with the same asymptotic complexity. This is a major advantage of their approach as 
compared with the preceding work by \cite{sfft}. 

The fast spherical Fourier transforms of \citeauthor{sfft2}\ were developed in several steps, which 
has led to different variants of the basic algorithm with different asymptotic complexities, ranging from 
$\mathcal{O}(L^{4})$, when using a separation of variables only, to $\mathcal{O}(L^{2}\log^{2\!} L)$, when using 
all techniques presented. We include the derivation of one particular variant, the \emph{seminaive} 
algorithm and its inverse, here. These seminaive algorithms are later used in the 
numerical experiment of Section \ref{sec:experiments}.

By \eqref{eq:spherical_harmonics}, a rearrangement of the right-hand side of \eqref{eq:S2_sampling_theorem} yields
\begin{equation}\label{eq:SFFT_rearranged}
\langle g, Y_{lm} \rangle_{\mathbb{S}^{2}} \hspace*{2pt}=\hspace*{2pt} M_{lm\!} \sum_{j=0}^{2L - 1} \!b_{\!j} \hspace*{1pt} P_{lm}(\cos\vartheta_{j})\! \sum_{k=0}^{2L - 1}\! g(\vartheta_{j}, \varphi_{k}) \hspace*{1pt} \mathrm{e}^{-\mathrm{i} m \varphi_{k}\!}, \:~~~~~ |m| \leq l < L,
\end{equation}
denoting by $M_{lm}$ the normalization constant of the spherical harmonic $Y_{lm}$. 
This separation of variables reduces the asymptotic complexity of the naive spherical Fourier 
transform from $\mathcal{O}(L^{5})$ to $\mathcal{O}(L^{4})$, as indicated above.

Precomputation of the inner sum in \eqref{eq:SFFT_rearranged} for all $-L < m < L$ can be done in 
$\mathcal{O}(L \log L)$ steps for each $j$ by using a standard Cooley-Tukey FFT (see 
\cite[Sect.\ 30.2]{algorithms}, for example). This results in an asymptotic complexity of $\mathcal{O}(L^{2} \log L)$ 
for this first step, while the total asymptotic complexity of $\mathcal{O}(L^{4})$ remains the same.

The central tool in the fast transforms of \citeauthor{sfft2}\ is a fast \emph{discrete Legendre transform} (DLT), 
i.e., a tool to compute the collection of sums
\begin{equation}\label{eq:DLT}
M_{lm\!} \sum_{j=0}^{2L - 1}\! b_{\!j} \hspace*{1pt} P_{lm}(\cos\vartheta_{j}) \hspace*{1pt} t_{j}, ~~~~~~ l = |m|, \dots, L - 1,
\end{equation}
for a fixed $- L < m < L$, $[t_{j}]_{j = 0,\dots,2L - 1}$ being arbitrary 
input data. Such fast DLT can be used to evaluate the outer sum in \eqref{eq:SFFT_rearranged}, running 
through all $-L < m < L$.

In the seminaive algorithm, the asymptotic complexity of the naive DLT is reduced by a fast 
\emph{discrete cosine transform} (DCT). The general DCT is defined as follows (cf.\ \cite[Sect.\ 5.6]{algorithms2}): Let 
$\vec{u} \coloneqq [u_{j}]_{j=0,\dots,N-1\!}$ be some data of length $N \!\in \mathbb{N}$ and set, in this subsection 
only, $\vartheta_{j} \coloneqq (2j + 1) \pi / 2N$. Any method for computation of the matrix-vector product
\begin{equation}\label{eq:DCT}
\underbrace{
\begin{bmatrix}
\sqrt{1/N}\!\! &                    &                &                \\
               & \!\!\sqrt{2/N}\!\! &                &                \\
               &                    & \!\!\ddots\!\! &                \\
               &                    &                & \!\!\sqrt{2/N} 
\end{bmatrix}
\cdot
\begin{bmatrix}
1                            & \cdots & 1                              \\
\cos \vartheta_{0}           & \cdots & \cos \vartheta_{N-1}           \\
\vdots                       &        & \vdots                         \\
\cos((N\!-\!1)\vartheta_{0}) & \cdots & \cos((N\!-\!1)\vartheta_{N-1})
\end{bmatrix}
}_{\hspace*{12pt}\eqqcolon\,C_{N}}
\cdot\hspace*{3.5pt}
\vec{u}
\end{equation}
is referred to as a DCT. Such computation is apparently associated with an asymptotic complexity of
$\mathcal{O}(N^{2})$ if no fast algorithm is used. By a factorization of the DCT matrix $C_{N}$, 
the asymptotic complexity of the naive DCT can be reduced to $\mathcal{O}(N \log N)$ (see 
\citep{dct} or \citep[Sect.\ 5.6]{algorithms2}, for example).

Two properties of such DCT are particularly important in our context: Firstly, the DCT matrix $C_{N}$ 
is \emph{orthogonal}. If $\vec{v}$ and $\vec{w}$ are two vectors of length 
$N\!$, and $\langle\cdot,\cdot\rangle_{\mathbb{C}^{N\!}}$ denotes the standard Euclidean inner product, this 
means that $\langle \vec{v}, \vec{w} \rangle_{\mathbb{C}^{N}} = \langle C_{N} 
\vec{v}, C_{N} \vec{w} \rangle_{\mathbb{C}^{N\!}}$. Secondly, if \hspace*{1pt}$\vec{p} \coloneqq 
[p(\vartheta_{j})]_{j=0,\dots,N-1}$, $p$ being an arbitrary trigonometric polynomial of degree at most 
$N\!$, then the elements $[C_{N} \vec{p}\hspace*{0.5pt}]_{j\!}$ vanish for $j > \textnormal{deg}(p)$.

We now consider the case $m = 0$ in \eqref{eq:DLT}; all other cases can be treated 
similarly. Choose $N \!= 2L$ above, and set 
\begin{equation*}
\vec{t} \coloneqq [b_{\!j} \hspace*{1pt} t_{j}]_{j=0,\dots,2L-1} ~~~~ \textnormal{and} ~~~~
\vec{P}_{\!l} \coloneqq M_{l,0} \cdot [P_{l,0}(\cos\vartheta_{j})]_{j=0,\dots,2L-1}, ~~~~ l < L.
\end{equation*}
Computation of the collection of sums \eqref{eq:DLT} then amounts to the computation of the inner product 
$\langle \vec{t}, \vec{P}_{\!l} \rangle_{\mathbb{C}^{2L\!}}$ for each $l$.
Since $P_{l,0}(\cos\vartheta)$ is a trigonometric polynomial of degree $l$, we have that 
\begin{equation}\label{eq:dct_trick}
\langle \vec{t}, \vec{P}_{\!l} \rangle_{\mathbb{C}^{2L\!}} 
\hspace*{1pt}=\hspace*{1pt} \langle C_{2L} \vec{t}, C_{2L} \vec{P}_{\!l} \rangle_{\mathbb{C}^{2L\!}}
\hspace*{1pt}=\hspace*{0pt} \!\sum_{j=0}^{2L - 1} [C_{2L} \vec{t}]_{j} \hspace*{2pt} [C_{2L}  \vec{P}_{\!l}]_{j}
\hspace*{1pt}=\hspace*{1pt} \!\sum_{j=0}^{l}\hspace*{2.5pt} [C_{2L} \vec{t}]_{j} \hspace*{2pt} [C_{2L} \vec{P}_{\!l}]_{j}.
\end{equation}

Equation \eqref{eq:dct_trick} shows that the inner product $\langle \vec{t}, 
\vec{P}_{\!l} \rangle_{\mathbb{C}^{2L}}$ can be computed in $l$ steps instead of $2L - 1$, if the vectors 
$C_{2L} \vec{t}$ and $C_{2L} \vec{P}_{\!l}$ are readily available. When this approach is used 
for all $m$, this does not yet change the asymptotic complexity, but reduces the total required computation 
work significantly for sufficiently large $L$. A truly \emph{fast} DLT can now be obtained in the following way: Let 
$\vec{P}_{\!lm} \coloneqq M_{lm\!} \cdot [P_{lm}(\cos\vartheta_{j})]_{j=0,\dots,2L-1}$, 
$|m| \leq l < L$. Since the vectors $C_{2L} \vec{P}_{\!lm\!}$ do not depend on the input data, and a 
significantly large part of their elements are zero, we may assume them to be stored and readily available during 
runtime. This results in a fast DLT with an asymptotic complexity of $\mathcal{O}(L^{2})$ instead of 
$\mathcal{O}(L^{3})$, and, hence, in a fast spherical Fourier transform with an asymptotic
complexity of $\mathcal{O}(L^{3})$ instead of $\mathcal{O}(L^{4})$.

To close this subsection, we revise the derivation of the inverse seminaive spherical Fourier 
transform. The task is to reconstruct the function values $g(\vartheta_{j},\varphi_{k})$ from the spherical
Fourier coefficients $\langle g, Y_{lm} \rangle_{\mathbb{S}^{2}}$. For this purpose, 
\citeauthor{sfft2}\ again use a separation of variables to get
\begin{equation}\label{eq:iSFFT_rearranged}
g(\vartheta_{j},\varphi_{k}) \hspace*{2pt}= \hspace*{-7pt}\sum_{m=-L+1}^{L-1}\hspace*{-9pt} \mathrm{e}^{\mathrm{i} m \varphi_{k}}\! \sum_{l= |m|}^{L-1} \!M_{lm} \hspace*{1pt} P_{lm}(\cos\vartheta_{j}) \hspace*{1pt} \langle g, Y_{lm} \rangle_{\mathbb{S}^{2}}, \:~~~ j,k = 0,\dots,2L-1.
\end{equation}

\begin{sloppypar}
It is clear that the function values $g(\vartheta_{j},\varphi_{k})$ can be reconstructed from
the collection of inner sums in \eqref{eq:iSFFT_rearranged} by means of a standard iFFT, 
which has the same asymptotic complexity $\mathcal{O}(L \log L)$ as the forward transform (again, 
see \cite[Sect.\ 30.2]{algorithms}).
\end{sloppypar}

We again consider the case $m = 0$ only. The collection of inner sums in \eqref{eq:iSFFT_rearranged}
can then be written as the matrix-vector product
\begin{equation*}
\underbrace{
\begin{bmatrix}
P_{0,0}(\cos\vartheta_{0})    & \cdots & P_{L-1,0}(\cos\vartheta_{0})    \\
\vdots                        &        & \vdots                          \\
P_{0,0}(\cos\vartheta_{2L-1}) & \cdots & P_{L-1,0}(\cos\vartheta_{2L-1})
\end{bmatrix}
\cdot
\begin{bmatrix}
M_{0,0}\!\! &                &               \\
            & \!\!\ddots\!\! &               \\
            &                & \!\!M_{L-1,0}
\end{bmatrix}
%
}_{\hspace*{11pt}\eqqcolon\,P_{0}^{\mathrm{T}}}
\cdot\hspace*{3.5pt}
[\langle g, Y_{l,0} \rangle_{\mathbb{S}^{2}}]_{l=0,\dots,L-1}.
\end{equation*}
Due to Theorem \ref{thm:S2_sampling_theorem} and the orthonormality of the spherical harmonics,
the non-transposed matrix $P_{0}$ is associated with the forward transform (we encounter the same 
phenomenon in the upcoming Sections \ref{subsubsec:DRT} and \ref{subsec:algebra}). In particular,
$P_{0}$ represents the forward DLT when dropping the weights $b_{j\!}$ (cf.\ \eqref{eq:DLT}). 
The above discussion shows that $P_{0}$ possesses the factorization
\begin{equation*}
P_{0} \hspace*{2pt}=\hspace*{2pt} [C_{2L}\vec{P}_{0,0}\hspace*{1pt},\dots,C_{2L}\vec{P}_{\!L-1,0}]^{\mathrm{T}\!} \hspace*{0.5pt}\cdot\hspace*{1pt} C_{2L}. 
\end{equation*}
This immediately reveals
\begin{equation*}
P_{0}^{\mathrm{T}\!} \hspace*{2pt}=\hspace*{2pt} C_{2L\!}^{\mathrm{T}} \hspace*{1pt}\cdot\hspace*{1pt} [C_{2L}\vec{P}_{0,0}\hspace*{1pt},\dots,C_{2L}\vec{P}_{\!L-1,0}]
\end{equation*}
for the inverse transform.

At this point, we recall the identity $C_{2L\!}^{\mathrm{T}\!} = C_{2L}^{-1}$, i.e., the 
orthogonality of the DCT matrix $C_{2L}$. By the use of a fast inverse DCT (iDCT) with an asymptotic 
complexity of $\mathcal{O}(L \log L)$ (again, see \citep{dct} or \cite[Sect.\ 5.6]{algorithms2}), the 
same ideas as above now easily yield the inverse seminaive spherical Fourier transform of \citeauthor{sfft2}, 
which has the same asymptotic complexity $\mathcal{O}(L^{3})$ as the forward transform.

\subsubsection{Completion: discrete $\vec{R}$ transforms}\label{subsubsec:DRT}

We now discuss our discrete $R$ transform (DRT) and its inverse (iDRT) to finalize the above-described 
class of fast SGL Fourier transforms. To this end, let $\vec{s} \coloneqq [s_{i}]_{i=0,\dots,2B-1}$
be some input data. For a fixed $l$, we bring the right-hand side of \eqref{eq:DRT} into matrix-vector 
notation:\vspace*{-4pt}
\begin{equation}\label{eq:DRT_matrix}
\left\lbrace\! 
\vphantom{
\begin{bmatrix}
N_{l+1,l}\hspace*{-6pt} &                                    &                       \\
                        & \hspace*{-6pt}\ddots\hspace*{-2pt} &                       \\
                        &                                    & \hspace*{-2pt}N_{B,l} 
\end{bmatrix}
}
\right.\!\!\underbrace{\!
\begin{bmatrix}
N_{l+1,l}\hspace*{-6pt} &                                    &                       \\
                        & \hspace*{-6pt}\ddots\hspace*{-2pt} &                       \\
                        &                                    & \hspace*{-2pt}N_{B,l} 
\end{bmatrix}
\!\!\!\!\hspace*{1pt}
\begin{bmatrix}
R_{l+1,l}(r_{0})\! & \cdots & \!R_{l+1,l}(r_{2B-1}) \\
\vdots             &        & \vdots                \\
R_{B,l}(r_{0})     & \cdots & R_{B,l}(r_{2B-1}) 
\end{bmatrix}\!
}_{\hspace*{10pt}\eqqcolon\,R_{l}}
\!\!\hspace*{2pt}
\overbrace{\!
\begin{bmatrix}
\mathrm{e}^{-r_{0}^{2}}\hspace*{-4pt} &                                    &                                            \\
                                      & \hspace*{-4pt}\ddots\hspace*{-3pt} &                                            \\
                                      &                                    & \hspace*{-3pt}\mathrm{e}^{-r_{2B-1\!}^{2}} 
\end{bmatrix}\!}^{\hspace*{8pt}\eqqcolon\,E}
\!\! \left.
\vphantom{
\begin{bmatrix}
N_{l+1,l}\!\!\! &                    &               \\
                & \!\!\!\ddots\!\!\! &               \\
                &                    & \!\!\!N_{B,l} 
\end{bmatrix}
}
\!\right\rbrace 
\!\!\big\lbrace \textnormal{diag}[\tilde{a}_{i}]_{i=0,\dots,2B - 1\!}
\cdot
\vec{s} \big\rbrace\!.
\end{equation}
Running the forward DRT of order $l$, which we now rigorously define as the evaluation of \eqref{eq:DRT_matrix},
can be done in $\mathcal{O}(B^{2})$ steps by using the \emph{Clenshaw algorithm} \citep{clenshaw}.
This is possible, since the Laguerre polynomial $L_{n-l-1}^{(l+1/2)\!}$ is included in the radial part $R_{nl}$ 
of the SGL basis functions as a factor, and the radial functions $R_{nl}$ thus satisfy the following three-term 
recurrence relation:
\begin{lemma}
Let \hspace*{1pt}$0 \leq l < B$ and $r \in [0,\infty)$ be given. Then 
\begin{equation*}
R_{n+1,l}(r) \hspace*{2pt}=\hspace*{2pt} \frac{2n - l - 1/2 - r^{2}}{\sqrt{(n + 1/2)(n-l)}} R_{nl}(r) - \sqrt{\frac{(n - 1/2)}{(n + 1/2)}\frac{(n-l-1)}{(n-l)}} R_{n-1,l}(r)
\end{equation*}
for $n > l$, and 
\begin{equation*}
R_{l+1,l}(r) \hspace*{2pt}=\hspace*{2pt} \sqrt{\frac{2}{\Gamma(l + 3/2)}} \hspace*{1pt} r^{l\!}, ~~~~~~ R_{ll} \hspace*{2pt}\equiv\hspace*{2pt} 0. 
\end{equation*} 
\end{lemma}\vspace*{4pt}

In view of the iDRT, we observe that for all $0 \leq m,n \leq B-l$,
\begin{align*}
\left[R_{l} \cdot \textnormal{diag}[a_{i} \hspace*{1pt} r_{i}^{2}]_{i=0,\dots,2B - 1\!} \cdot R_{l}^{\mathrm{T}}\hspace*{1pt} \right]_{mn\!}
\hspace*{2pt}&=\hspace*{2pt} \!\sum_{k = 0}^{2B - 1}\! a_{k} \hspace*{1pt} r_{k}^{2} \hspace*{1.5pt} N_{l+m,l} \hspace*{0pt} \overbrace{R_{l+m,l}}^{\hspace*{-20pt}\hspace*{7pt}\textnormal{deg}\:\leq\:2B-2\hspace*{-20pt}}\hspace*{0.5pt}(r_{k}) \hspace*{1pt} N_{l+n,l} \hspace*{0pt} \overbrace{R_{l+n,l}}^{\hspace*{-20pt}\hspace*{7pt}\textnormal{deg}\:\leq\:2B-2\hspace*{-20pt}}\hspace*{0.5pt}(r_{k})\\
\hspace*{2pt}&=\hspace*{2pt} \!\int_{0}^{\infty}\!\! N_{l+m,l} \hspace*{1pt} R_{l+m,l}\hspace*{0.5pt}(r) \hspace*{1pt} N_{l+n,l} \hspace*{1pt} R_{l+n,l}\hspace*{0.5pt}(r) \hspace*{1pt} r^{2} \hspace*{1pt} \mathrm{e}^{-r^{2\!}} \hspace*{1pt} \mathrm{d}r\\[8pt]
\hspace*{2pt}&=\hspace*{2pt} \delta_{mn},
\end{align*}
by Theorem \ref{thm:half-range_Hermite} and the orthogonality of the polynomials $R_{l+m,l}$ and $R_{l+n,l}$ 
(cf.\ Section \ref{sec:sgl}). Note that we dropped the factors $\exp(r_{i}^{2})$ and $\exp(-r_{i}^{2})$ here 
(again, see Section \ref{sec:discussion}). Hence,
\begin{equation*}
R_{l}^{\mathrm{T}\!} \hspace*{2pt}=\hspace*{2pt} \left\lbrace R_{l} \cdot \textnormal{diag}[a_{i} \hspace*{1pt} r_{i}^{2}]_{i=0,\dots,2B - 1} \right\rbrace^{\!-1} \vspace*{4pt} \hspace*{-11pt}.
\end{equation*}
For a given data vector $\vec{t} \coloneqq [t_{k}]_{k=0,...,B-l}$, we thus define the iDRT of order $l$ as the 
evaluation of the product $R_{l}^{\mathrm{T}\!} \cdot \vec{t}$. This can be done in $\mathcal{O}(B^{2})$ steps, 
just as one run of the forward transform (cf.\ Section \ref{subsec:algebra}).

Summarizing all of the above, we state the layout of our FSGLFTs and iFSGLFTs as Algorithm \ref{alg:FSGLFT} 
and \ref{alg:iFSGLFT}, respectively. Note that all of these transforms have an asymptotic complexity of 
$\mathcal{O}(B^{4})$ instead of the naive $\mathcal{O}(B^{7})$.

\begin{samepage}
\begin{algorithm}[t]\vspace*{4pt}
 \KwData{Sample values $f(r_{i}, \vartheta_{j}, \varphi_{k})$; $i,j,k = 0,\dots,2B\!-\!1$, of a function $f$ with bandlimit $B\!\hspace*{1pt} \in \mathbb{N}$}\vspace*{3pt}
 \KwResult{SGL Fourier coefficients $\hat{f}_{nlm}$, $|m| \leq l < n \leq B$}\vspace*{3pt}
 \For{$i = 0$ \KwTo $2B - 1$}
 {
    \vspace*{3pt}
    Compute Fourier coefficients $\langle f(r_{i},\cdot,\cdot), Y_{lm} \rangle_{\mathbb{S}^{2}}$, $|m| \leq l < B$, from function samples $f(r_{i},\vartheta_{j},\varphi_{k})$; $j,k=0,\dots,2B-1$, by using a fast spherical Fourier transform;\vspace*{3pt}
 }\vspace*{5pt}
 \For{$m = 1 - B$ \KwTo $B - 1$}
 {
    \For{$l = |m|$ \KwTo $B - 1$}
    {
       \vspace*{3pt}
       Compute SGL Fourier coefficients $[\hat{f}_{nlm}]_{n=l+1,\dots,B}$ by using the DRT in\vspace*{-4pt}
       \begin{equation*}
          [\hat{f}_{nlm}]_{n=l+1,\dots,B} = \lbrace \!\hspace*{1pt} R_{l} \!\hspace*{1.5pt}\cdot\!\hspace*{1.5pt} E \rbrace \!\hspace*{1pt}\cdot\!\hspace*{1pt} \big\lbrace \textnormal{diag}[\tilde{a}_{i}]_{i=0,\dots,2B - 1\!}\hspace*{1pt} \cdot [\langle f(r_{i},\cdot,\cdot), Y_{lm} \rangle_{\mathbb{S}^{2}}]_{i=0,\dots,2B-1}\big\rbrace;\vspace*{-6pt}
       \end{equation*}
    }
 }
 \caption{Prototypical FSGLFT}\label{alg:FSGLFT}
\end{algorithm}
\begin{algorithm}[t]\vspace*{4pt}
 \KwData{SGL Fourier coefficients $\hat{f}_{nlm}$, $|m| \leq l < n \leq B$, of a function $f$ with bandlimit $B\!\hspace*{1pt} \in \mathbb{N}$}\vspace*{3pt}
 \KwResult{Function values $f(r_{i}, \vartheta_{j}, \varphi_{k})$; $i,j,k = 0,\dots,2B-1$}\vspace*{3pt}
 \For{$m = 1 - B$ \KwTo $B - 1$}
 {
    \For{$l = |m|$ \KwTo $B - 1$}
    {
       \vspace*{3pt}
       Reconstruct spherical Fourier coefficients $\langle f(r_{i},\cdot,\cdot), Y_{lm} \rangle_{\mathbb{S}^{2}}$, $i = 0, \dots, 2B - 1$, by using the iDRT in\vspace*{-4pt}
       \begin{equation*}
          [\langle f(r_{i},\cdot,\cdot), Y_{lm} \rangle_{\mathbb{S}^{2}}]_{i=0,\dots,2B-1} = R_{l}^{\mathrm{T}} \cdot [\hat{f}_{nlm}]_{n=l+1,\dots,B};\vspace*{-6pt}
       \end{equation*}
    }
 }\vspace*{5pt}
 \For{$i = 0$ \KwTo $2B - 1$}
 {
    \vspace*{3pt}
    Reconstruct function values $f(r_{i},\vartheta_{j},\varphi_{k})$; $j,k=0,\dots,2B-1$, from Fourier coefficients $\langle f(r_{i},\cdot,\cdot), Y_{lm} \rangle_{\mathbb{S}^{2}}$, $|m| \leq l < B$, by using a fast inverse spherical Fourier transform;\vspace*{3pt}
 }
 \caption{Prototypical iFSGLFT} \label{alg:iFSGLFT}
\end{algorithm}
\end{samepage}

\subsection{Matrix-vector notation of the transforms}\label{subsec:algebra}

In this section, we give a description of the transforms presented in Sections \ref{subsec:DSGLFT} and \ref{subsec:FSGLFT} 
in terms of matrix-vector products. We shall use the standard \emph{Kronecker product}, denoted by $\otimes$. 
To simplify the notation further, for a fixed bandlimit $B \in \mathbb{N}$, we introduce the linear indices
%
\begin{align*}
\mu \hspace*{2pt}&=\hspace*{2pt} \mu(j,k) \hspace*{2pt}\coloneqq\hspace*{2pt} 2Bj + k,\\
\psi \hspace*{2pt}&=\hspace*{2pt} \psi(i,j,k) \hspace*{2pt}\coloneqq\hspace*{2pt} 4B^{2}i + 2Bj + k,\\
\nu \hspace*{2pt}&=\hspace*{2pt} \nu(l,m) \hspace*{2pt}\coloneqq\hspace*{2pt} l(l+1) + m,\\
\omega \hspace*{2pt}&=\hspace*{2pt} \omega(n,l,m) \hspace*{2pt}\coloneqq\hspace*{2pt} n (n-1) (2n - 1) / 6 + l(l+1) + m,
\end{align*}
%
so that $\mu$ enumerates the sampling angles $\vec{a}_{\mu} \coloneqq (\vartheta_{j},\varphi_{k})$
and corresponding weights $c_{\mu\!} \coloneqq b_{\!j}$; $j,k = 0,\dots,2B-1$, $\psi$ enumerates the sampling points 
$\vec{x}_{\psi} \coloneqq (r_{i},\vartheta_{j},\varphi_{k})$ and corresponding weights $w_{\psi} \coloneqq 
\tilde{a}_{i}\hspace*{1pt}b_{\!j}$; $i,j,k = 0,\dots,2B-1$, $\nu$ enumerates the spherical harmonics $Y_{\nu} \coloneqq 
Y_{lm}$, $|m| \leq l < B$, while $\omega$ enumerates the SGL basis functions $H_{\omega} \coloneqq H_{nlm}$, 
$|m| \leq l < n \leq B$. The indices of the rows and columns of a matrix shall be separated by a semicolon.
Because the following considerations are mainly theoretical, we omit all brackets solely relevant in practice, 
i.e., those specifying the order of operations.

Let a function $f$ with bandlimit $B \in \mathbb{N}$ and SGL Fourier coefficients $\hat{f}_{\omega}
\coloneqq \hat{f}_{nlm}$ be given. Furthermore, let $\Psi \coloneqq 8B^{3\!}$ denote the total
number of sample points $\vec{x}_{\psi}$, and let $\Omega \coloneqq B (B+1) (2B+1)/6$ denote the total number of SGL 
Fourier coefficients $\hat{f}_{\omega}$ to be computed. We define the sample vector $\vec{f} \coloneqq [f(\vec{x}_{\psi})
]_{\psi = 0,\dots,\Psi - 1}$ and the vector $\hat{\vec{f}\hspace*{2pt}}\hspace*{-2pt} \coloneqq [\hat{f}_{\omega}]_{\omega = 0,\dots,\Omega-1}$ 
containing the SGL Fourier coefficients of $f$. The naive DSGLFT/iDSGLFT of Section \ref{subsec:DSGLFT} can thus be restated
as
\begin{align*}
\hat{\vec{f}\hspace*{2pt}}\hspace*{-2pt} \hspace*{2pt}&=\hspace*{2pt} [\overline{H_{\omega}(\vec{x}_{\psi})} \exp(-|\vec{x}_{\psi}|^{2})]_{\omega=0,\dots,\Omega-1;\hspace*{1pt}\psi=0,\dots,\Psi-1} \cdot \textnormal{diag}[w_{\psi}]_{\psi=0,\dots,\Psi-1} \cdot \hspace*{1pt}\vec{f},\\[6pt]
\vec{f} \hspace*{2pt}&=\hspace*{2pt} [\overline{H_{\omega}(\vec{x}_{\psi})}]_{\omega=0,\dots,\Omega-1;\hspace*{1pt}\psi=0,\dots,\Psi-1}^{\mathrm{H}} \cdot\hspace*{0pt} \hspace*{1pt}\hat{\vec{f}\hspace*{2pt}}\hspace*{-2pt}\!\hspace*{0.5pt},
\end{align*}
respectively. Note that the quadrature weights $w_{\psi}$ appear only in the forward transform.
Furthermore, it becomes apparent that the inverse transform is represented by precisely the Hermitean 
transpose of the matrix associated with the forward transform, when the weighting of the SGL basis 
function samples $H_{\omega}(\vec{x}_{\psi})$ and of the function samples $\vec{f}$ is dropped 
(compare also Sections \ref{subsubsec:SFFT} and \ref{subsubsec:DRT}). This is a consequence of the 
orthonormality of the SGL basis functions and our SGL sampling theorem. Recall that the computation 
of each element of the transformation matrices requires $\mathcal{O}(B)$ steps. Because the size of 
these matrices is both $B(B+1)(2B+1)/6 \times 8B^{3\!}$, this gives the total asymptotic complexity 
of $\mathcal{O}(B^{7})$ of both forward and inverse transform.

Let the matrices $R_{l}$, $0 \leq l < B$, and $E$ be defined as in \eqref{eq:DRT_matrix}. Set further 
\begin{alignat*}{2}
&\hspace*{1pt}Y &&\coloneqq\hspace*{2pt} [\overline{Y_{\nu}(\vec{a}_{\mu})}]_{\nu=0,\dots,B^{2}-1;\hspace*{1pt}\mu=0,\dots,4B^{2}-1},\\
&A \hspace*{2pt}&&\coloneqq\hspace*{2pt} \textnormal{diag}[\tilde{a}_{i}]_{i=0,\dots,2B-1}, \vphantom{[\overline{Y_{\nu}(\vec{a}_{\mu})}]_{\nu=0,\dots,B^{2}-1;\hspace*{1pt}\mu=0,\dots,4B^{2}-1}}\\
&C \hspace*{2pt}&&\coloneqq\hspace*{2pt} \textnormal{diag}[c_{\mu}]_{\mu=0,\dots,4B^{2}-1}\vphantom{[\overline{Y_{\nu}(\vec{a}_{\mu})}]_{\nu=0,\dots,B^{2}-1;\hspace*{1pt}\mu=0,\dots,4B^{2}-1} \mathrm{e}^{r_{i}^{2}}]_{i=0,\dots,2B-1}},
\end{alignat*}
and let $I_{N}$ denote the $N \times N$ identity matrix ($N\! \in \mathbb{N}$). Note that $B^{2}$ is the total number of spherical harmonics
$Y_{lm}$, $|m| \leq l < B$, while $4B^{2}$ is the total number of sampling angles $(\vartheta_{j},\varphi_{k})$; 
$j,k = 0,\dots,2B-1$. After the separation of variables described in Section \ref{subsec:FSGLFT}, we find that
\begin{alignat*}{2}
&\hat{\vec{f}\hspace*{2pt}}\hspace*{-2pt} \hspace*{2pt}=\hspace*{2pt} P \hspace*{2pt}\cdot 
&&\overbrace{\!\begin{bmatrix}
                            &                &                             \\[-10pt]
\boxed{\tilde{R}_{1-B}}\!\! &                &                             \\[-4pt]
                            & \!\!\ddots\!\! &                             \\
                            &                & \!\!\boxed{\tilde{R}_{B-1}} \\[5.5pt]
\end{bmatrix}\!}^{\hspace*{-20pt}2B - 1 \textnormal{ blocks (see below)}\hspace*{-20pt}}
\hspace*{1pt}\cdot\hspace*{1pt} \overbrace{\lbrace I_{B^{2\!}} \otimes E \rbrace}^{\textnormal{diagonal}} 
\hspace*{1pt}\cdot\hspace*{1pt}  \overbrace{\lbrace I_{B^{2\!}} \otimes A \rbrace}^{\textnormal{diagonal}}
\hspace*{1pt}\cdot \hspace*{3pt} Q \cdot
\overbrace{\!\begin{bmatrix}
              &                &               \\[-10.5pt]
\boxed{Y}\!\! &                &               \\[-4pt]
              & \!\!\ddots\!\! &               \\
              &                & \!\!\boxed{Y} \\[3.5pt]
\end{bmatrix}\!}^{\hspace*{-30pt}2B \textnormal{ blocks of size } B^{2} \times \hspace*{1pt} 4B^{2}\hspace*{-30pt}}
\cdot \hspace*{1pt} \overbrace{\lbrace I_{2B} \otimes C \rbrace}^{\textnormal{diagonal}} \hspace*{1pt}\cdot\hspace*{2pt} \vec{f},
\\
&\vec{f} \hspace*{9pt}=\hspace*{2pt} 
&&\!\begin{bmatrix}
                             &                &                              \\[-10pt]
\boxed{Y^{\mathrm{H}\!}}\!\! &                &                              \\[-4pt]
                             & \!\!\ddots\!\! &                              \\
                             &                & \!\!\boxed{Y^{\mathrm{H}\!}} \\[3pt]
\end{bmatrix}\!
\cdot \hspace*{1pt} Q^{\mathrm{T}\!} \cdot \! 
\begin{bmatrix}
                                           &                &                                            \\[-10pt]
\boxed{\tilde{R}_{1-B}^{\mathrm{T}\!}}\!\! &                &                                            \\[-4pt]
                                           & \!\!\ddots\!\! &                                            \\
                                           &                & \!\!\boxed{\tilde{R}_{B-1}^{\mathrm{T}\!}} \\[7pt]
\end{bmatrix}\!
\cdot \hspace*{1pt}P^{\mathrm{T\!}} \cdot \hspace*{1.5pt} \hat{\vec{f}\hspace*{2pt}}\hspace*{-2pt}\!\hspace*{0.5pt},
\end{alignat*}
where the matrices $\tilde{R}_{1-B}, \dots, \tilde{R}_{B-1}$ have again a block structure,
\begin{equation*}
\tilde{R}_{m} \hspace*{2pt}=\hspace*{2pt} 
\underbrace{\!\begin{bmatrix}
                &                &                 \\[-10pt]
\boxed{R_{|m|}} &                &                 \\[-4pt]
                & \!\!\ddots\!\! &                 \\
                &                & \boxed{R_{B-1}} \\[7pt]
\end{bmatrix}\!}_{\hspace*{-22pt}B - |m| \textnormal{ blocks $R_{l}$ of size } B-l \hspace*{1pt}\times\hspace*{1pt} 2B\hspace*{-22pt}}\hspace*{0.5pt},
\end{equation*}
while $P\!$ and $Q$ are suitable permutation matrices. We introduce these permutation matrices here for an 
improved structural depiction; they do not change the asymptotic complexity. Note that the separation variables
results in a factorization of the transformation matrices. The asymptotic complexity is reduced to $\mathcal{O}(B^{6})$ 
by evaluating the matrix-vector products successively. The factorization of the matrix of the inverse 
transform is obtained by taking the Hermitean transpose of the factorized matrix of the forward transform
and dropping the diagonal weight matrices (compare with Section \ref{subsubsec:SFFT}).

Finally, as explained in Section \ref{subsec:FSGLFT}, a fast spherical Fourier transform and inverse 
can be used to reduce the asymptotic complexity to $\mathcal{O}(B^{5})$. This amounts to a factorization 
of the matrices $Y\!$ and $Y^{\mathrm{H}\!}$. Employing the Clenshaw algorithm in the DRT and iDRT further 
reduces the asymptotic complexity to $\mathcal{O}(B^{4})$. This amounts to a factorization of the 
matrices $R_{l}$ and $R_{l}^{\mathrm{T}\!}$, $0 \leq l < B$.

\section{Numerical experiment}\label{sec:experiments}
We realized the naive DSGLFT and iDSGLFT of Section \ref{subsec:DSGLFT}, as well as the FSGLFT and iFSGLFT 
of Section \ref{subsec:FSGLFT} in MathWorks' MATLAB R2015a. For the spherical subtransform in the FSGLFT/iFSGL\-FT, 
we implemented the seminaive spherical Fourier transform and inverse of \cite{sfft2}, described in 
Section \ref{subsubsec:SFFT}, using the built-in FFT and inverse as well as the built-in fast DCT and inverse. 
No parallelization was done.

In both the DSGLFT/iDSGLFT and the FSGLFT/iFSGLFT, we precomputed the sampling radii $r_{i}$ and transformed 
sampling angles $(\cos\vartheta_{j},\varphi_{k})$ for the bandlimits listed below with high precision in Wolframs' 
Mathematica 10, and stored them double format. We did the same for the corresponding spherical quadrature weights 
$b_{j}$ and the modified radial quadrature weights $\tilde{a}_{i}$. 
For the seminaive spherical Fourier transform and inverse, we precomputed the transformed 
vectors $C_{2B} \vec{P}_{\!lm\!}$ (cf.\ Section \ref{subsubsec:SFFT}) for all bandlimits below in Matlab, 
and stored them in double format.

The actual testruns were performed on a Unix system with a 3.40 GHz Intel Core i7-3770 CPU. We iterated through 
the bandlimits $B = 2,4,8,16,32$. For each bandlimit, we generated random SGL Fourier coefficients $\hat{f}_{nlm}$. 
Both the real part and the imaginary part were uniformly distributed between $-1$ and $1$. We then performed the 
iDSGLFT as well as the iFSGLFT on these Fourier coefficients to reconstruct the corresponding function values 
$f(r_{i},\vartheta_{j},\varphi_{k})$. Subsequently, we transformed the function values back into SGL Fourier 
coefficients $\hat{f}_{nlm}^{\hspace*{0.5pt}\circ}$, using the DSGLFT and FSGLFT, respectively. We measured the total runtime 
of one forward and subsequent inverse transform, and the absolute and relative (maximum) transformation 
error,
\begin{equation*}
\max_{|m| \leq l < n \leq B} |\hat{f}_{nlm} - \hat{f}_{nlm}^{\hspace*{0.5pt}\circ}| ~~~~ \textnormal{and} ~~~~ \max_{|m| \leq l < n \leq B} \frac{|\hat{f}_{nlm} - \hat{f}_{nlm}^{\hspace*{0.5pt}\circ}|}{|\hat{f}_{nlm}|},
\end{equation*}
respectively.
We repeated the above procedure $10$ times, and determined the average runtime, the average absolute 
transformation error, and the average relative transformation error for each bandlimit.
We then performed the entire testrun again for $B\!\hspace*{1pt} = 64$ with the fast transforms. 

Table \ref{table:errors} shows the results of the error measurement. The results of the runtime measurement
are listed in Table \ref{table:runtime}.

\begin{table}[t]
\begin{minipage}{\textwidth}
\begin{center}
\begin{tabular}{ccccccc}
  \hline&&\\[-8pt]			
  \hspace*{5pt}$B$\hspace*{5pt} & \hspace*{5pt}iDSGLFT/DSGLFT\hspace*{5pt} & \hspace*{5pt}iFSGLFT/FSGLFT\hspace*{5pt} \\\hline
                                &                                          &                                          \\[-8pt]
    2 & (5.57\:$\pm$\:1.67)\:\small{E}\:\small$-16$ & (3.85\:$\pm$\:1.08)\:\small{E}\:\small$-16$ \\
    & (7.70\:$\pm$\:1.69)\:\small{E}\:\small$-16$ & (4.64\:$\pm$\:1.36)\:\small{E}\:\small$-16$ \\[3pt]
  4 & (1.35\:$\pm$\:0.25)\:\small{E}\:\small$-15$ & (8.45\:$\pm$\:1.23)\:\small{E}\:\small$-16$ \\
    & (3.32\:$\pm$\:2.10)\:\small{E}\:\small$-15$ & (2.23\:$\pm$\:1.60)\:\small{E}\:\small$-15$ \\[3pt]
  8 & (5.45\:$\pm$\:0.63)\:\small{E}\:\small$-15$ & (1.66\:$\pm$\:0.18)\:\small{E}\:\small$-15$ \\
    & (2.30\:$\pm$\:2.18)\:\small{E}\:\small$-14$ & (4.51\:$\pm$\:1.11)\:\small{E}\:\small$-15$ \\[3pt]
 16 & (2.01\:$\pm$\:0.33)\:\small{E}\:\small$-14$ & (3.96\:$\pm$\:0.51)\:\small{E}\:\small$-15$ \\
    & (1.99\:$\pm$\:1.25)\:\small{E}\:\small$-13$ & (2.98\:$\pm$\:1.31)\:\small{E}\:\small$-14$ \\[3pt]
 32 & (6.39\:$\pm$\:0.89)\:\small{E}\:\small$-14$ & (6.36\:$\pm$\:0.55)\:\small{E}\:\small$-15$ \\
    & (6.82\:$\pm$\:1.75)\:\small{E}\:\small$-13$ & (1.79\:$\pm$\:1.41)\:\small{E}\:\small$-13$ \\[3pt]
 64 & -                                           & (3.50\:$\pm$\:0.41)\:\small{E}\:\small$-14$ \\
    & -                                           & (8.45\:$\pm$\:2.87)\:\small{E}\:\small$-13$ \\[3pt]

  &&\\[-12pt]\hline
\end{tabular}\vspace*{4pt}
\end{center}
\end{minipage}
\caption{\textbf{(First row)} average absolute and \textbf{(second row)} average relative 
transformation error of one inverse and subsequent forward DSGLFT/FSGLFT, 
respectively.}\vspace*{6pt}\label{table:errors}
\end{table}

\begin{table}[t]
\begin{minipage}{\textwidth}
\begin{center}
\begin{tabular}{ccccccc}
  \hline&&\\[-8pt]			
  \hspace*{5pt}$B$\hspace*{5pt} & \hspace*{5pt}iDSGLFT/DSGLFT\hspace*{5pt} & \hspace*{5pt}iFSGLFT/FSGLFT\hspace*{5pt} \\\hline
                                &                                          &                                          \\[-8pt]
    2 & 2.34\:\small{E}\:\small$-2$\:s & 7.86\:\small{E}\:\small$-3$\:s \\[2pt]
  4 & 3.80\:\small{E}\:\small$-1$\:s & 2.93\:\small{E}\:\small$-2$\:s \\[2pt]
  8 & 1.56\:\small{E}\:\small$+1$\:s & 1.44\:\small{E}\:\small$-1$\:s \\[2pt]
 16 & 9.10\:\small{E}\:\small$+2$\:s & 8.68\:\small{E}\:\small$-1$\:s \\[2pt]
 32 & 5.27\:\small{E}\:\small$+4$\:s & 6.00\:\small{E}\:\small$+0$\:s \\[2pt]
 64 & -                              & 4.85\:\small{E}\:\small$+1$\:s \\[2pt]

  &&\\[-12pt]\hline
\end{tabular}\vspace*{4pt}
\end{center}
\end{minipage}
\caption{Average runtime of one inverse and subsequent forward DSGLFT/FSGLFT, 
respectively.}\vspace*{0pt}\label{table:runtime}
\end{table}

\section{Discussion, conclusions, and future developments}\label{sec:discussion}

As mentioned in Section \ref{sec:intro}, the SGL basis functions are nowadays used extensively in the simulation of \emph{biomolecular 
recognition processes}, such as protein-protein or protein-ligand docking. This is due to the existence 
of an elaborate machinery of fast \emph{SGL matching algorithms} (see \citep{gto, gto2}, and the references contained therein). All 
of these algorithms are spectral methods, i.e., they require the computation of the SGL Fourier coefficients of so-called 
\emph{affinity functions} prior to the actual (docking) simulation. This task is currently accomplished by sampling the affinity 
function $f$ of interest onto a regular Cartesian grid and using a \emph{midpoint method} for numerical integration:
\begin{equation*}
\hat{f}_{nlm} \hspace*{2pt}\approx\hspace*{2pt} \sum_{k} f(\vec{x}_{k}) \hspace*{1pt} \overline{H_{nlm}(\vec{x}_{k})} \hspace*{1pt} \Delta V, \vphantom{sum_{k}^{k}}
\end{equation*}
where $\vec{x}_{k}$ is the midpoint of the $k$th cell, and $\Delta V$ is the cell volume. 

While this approach is easily realized and useful for moderate problem sizes, it does benefit from the special structure 
\eqref{eq:sgl} of the SGL basis functions, and there is no guarantee for exactness. Our fast SGL Fourier transforms, 
on the other hand, crucially benefit from the special structure of the SGL basis functions and guarantee exactness in the
sense of our SGL sampling theorem (Theorem \ref{thm:SGL_sampling_theorem}). Specifically, the special structure of the SGL 
basis function allows to separate our discrete SGL Fourier transforms into a spherical and a radial subtransform, and, thus, 
to avoid computational redundancy to a large extent. 

The results in Tables \ref{table:errors} and \ref{table:runtime} clearly show that the FSGLFT and iFSGLFT
tested in Section \ref{sec:experiments} work very well for all bandlimits considered: The absolute and relative 
transformation errors are significantly smaller than those of the naive DSGLFT and iDSGLFT. This is
due to the smaller total number of operations in the fast transforms, resulting in less round-off error. The total runtime
of FSGLFT and iFSGLFT was significantly lower in all cases\:--\:even for the smallest bandlimits, which is typically not the case.
Since the bandlimits $B \leq 32$ are of most practical relevance, an interesting question for further research 
is how the SGL matching algorithms of Ritchie et al.\ perform in combination with these fast transforms. 

In the spherical subtransform of the fast SGL Fourier transforms tested in Section \ref{sec:experiments}, we used the 
seminaive fast spherical Fourier transform and inverse of \cite{sfft2}. The seminaive variant appears to be the optimal choice 
for bandlimits $B \leq 128$. For larger bandlimits, there are other variants described by \citeauthor{sfft2}, which 
should be considered. Note, however, that all of these variants result in the same asymptotic complexity $\mathcal{O}(B^{4})$
of our fast transforms. 

The spherical quadrature rules of \cite{sfft} (Theorem \ref{thm:S2_sampling_theorem}) and the Gaussian quadrature
rules (Theorem \ref{thm:half-range_Hermite}) yield an asymptotically optimal relation between the number of SGL Fourier
coefficients ($\mathcal{O}(B^{3})$) and sampling points on $\mathbb{R}^{3\!}$\hspace*{1pt} ($\mathcal{O}(B^{3})$). Spherical 
quadrature rules with a lower total number of sampling points and corresponding fast spherical Fourier transforms
are described in \citep{sfft4}. These fast transforms can easily be used in our framework as well, leaving the total
asymptotic complexity again untouched. 

In Section \ref{subsec:DSGLFT}, we introduced the factor $\exp(r_{i}^{2})$ to compensate for the fast decay of 
the quadrature weights $a_{i}$. This modification was accounted for by weighting the SGL basis function samples 
$H_{nlm}(r_{i},\vartheta_{j},\varphi_{k})$ by the factor $\exp(-r_{i}^{2})$, which was done during runtime
in Section \ref{sec:experiments}. We found that such adjustment is essential for bandlimits $B \geq 64$, when 
working with double precision. For bandlimits $B\! < 64$, the above modification results in slightly lower
transformation errors, which is why we used this adjustment consistently for all bandlimits. Of course, our approach 
requires the precomputation of the modified weights $\tilde{a}_{i}$ using a high precision; 
it does not affect the asymptotic complexity. 
Due to the absence of the weights $a_{i}$, there is generally no modification required in the inverse transforms.

For all bandlimits considered in this paper, the storage requirements of the precomputed data are not an issue. 
In the case $B = 64$, for example, the precomputed data for the FSGLFT and iFSGLFT of Section \ref{sec:experiments} 
require approximately 25\:MB of free disk space. For completeness, we record that the storage complexity of these 
fast transforms is $\mathcal{O}(B^{3})$ for both disk space and memory. This due to the precomputed data $C_{2B} 
\vec{P}_{lm}$ of Section \ref{subsubsec:SFFT}, which are stored on the disk and loaded during runtime, and the 
data in memory being processed. In general, the disk space requirements of our fast transforms are essentially the 
same as those of the particular spherical Fourier transforms employed.

As a more theoretical remark, we note that it is possible to obtain a true $\mathcal{O}(B^{3} \log^{2\!} B)$ 
FSGLFT/iFSGLFT by using an $\mathcal{O}(B^{2} \log^{2\!} B)$ variant of the spherical Fourier transforms of 
\citeauthor{sfft2}, and interchanging the Clenshaw algorithm in our discrete $R$ transform of Section 
\ref{subsubsec:DRT} with an $\mathcal{O}(B \log^{2\!} B)$ \emph{fast discrete polynomial transform} (refer to 
\citep{fdpt}, for instance). 
Note, however, that the smaller asymptotic complexity has to be traded with an increased storage complexity 
and a larger constant prefactor in runtime. Therefore, we do not expect a benefit of this approach for the 
bandlimits currently used in practice.

\begin{sloppypar}
Another interesting task for future research is to extend our fast SGL Fourier transforms in such a way that they 
can be applied to \emph{scattered} (i.e., non-gridded) data. This has already been achieved successfully in the 
classical FFT (see, e.g., \citep{nfft}), as well as in other generalized FFTs (see \citep{sfft3, sofft2}, for example).
\end{sloppypar}

Finally, we would like to emphasize that our fast SGL Fourier transforms are \emph{polynomial transforms}. 
With little adaptions, our approach can also be used for similar combinations of spherical harmonics and 
generalized Laguerre polynomials, such as those stated in \cite[Sect.\ 5.1.3]{dunkl_xu}, or the radially scaled
SGL basis functions of \citeauthor{gto}. Moreover, the underlying domain $\mathbb{R}^{3\!}$ of our transforms is 
non-compact. Notably, the experiments of Section \ref{sec:experiments} are one of the first performances of 
generalized FFTs on a non-compact domain. Other examples in this direction can be found in 
\citep{letzte}.

\begin{acknowledgement}
The authors would like to thank the referee for their valuable comments. Furthermore, the authors would like to thank 
Daniel Potts for pointing out the Clenshaw algorithm, Denis-Michael Lux for his assistance in programming, and Thomas 
Buddenkotte for scientific discussion. Finally, both authors are grateful to the Graduate School for Computing in 
Medicine and Life Sciences at the University of L\"ubeck. 
\end{acknowledgement}

\bibliography{references}
\bibliographystyle{spbasic}

\end{document}